\documentclass[12pt]{amsart}

\usepackage{amsmath,amscd,amssymb,amsfonts,amsthm,mathabx,hyperref}
\usepackage{xcolor}  
\hypersetup{colorlinks,   linkcolor={red!50!black},    citecolor={blue!70!black},   urlcolor={blue!80!black}}

\usepackage{tcolorbox}
\usepackage[cmtip,matrix,arrow]{xy}
\usepackage{tikz}
\usepackage{tikz-cd}

\newtheorem{theorem}{Theorem}[section]

\newtheorem{proposition}[theorem]{Proposition}

\newtheorem{lemma}[theorem]{Lemma}
\theoremstyle{definition}    
\newtheorem{definition}[theorem]{Definition}
\theoremstyle{remark}

\newtheorem{remark}[theorem]{Remark}
\newtheorem{remarks}[theorem]{Remarks}
\newtheorem{example}[theorem]{Example}

\newcommand\A{\mathcal{A}}

\newcommand\M{\mathcal{M}}

\renewcommand{\L}{\mathcal{L}}
\renewcommand{\O}{\mathcal{O}}

\newcommand{\ca}{\mathcal}

\newcommand{\R}{\mathbb{R}}

\newcommand{\Z}{\mathbb{Z}}

\renewcommand{\P}{\ca{P}}

\newcommand\lie[1]{\mathfrak{#1}}

\newcommand{\h}{\lie{h}}
\newcommand{\g}{\lie{g}}

\newcommand{\on}{\operatorname}

\newcommand{\Aut}{ \on{Aut} } 
\newcommand{\Gau}{ \on{Gau} }
 
\newcommand{\gau}{ \mf{gau} }
\newcommand{\Ad}{ \on{Ad} }

\newcommand{\Hom}{ \on{Hom}}

\renewcommand{\ker}{ \on{ker}}

\newcommand{\SO}{ \on{SO}}

\newcommand{\da}{\dasharrow}

\newcommand\qu{/\kern-.7ex/} 

\newcommand{\hra}{\hookrightarrow}

\renewcommand{\d}{{\mathsf{d}}}
\newcommand{\ol}{\overline}

\newcommand\eps{\epsilon}

\newcommand{\f}{\frac}

\newcommand{\p}{\partial}
\renewcommand{\l}{\langle}
\renewcommand{\r}{\rangle}
\newcommand\hh{{\f{1}{2}}}

\newcommand{\ti}{\tilde}
\newcommand{\eeq}{\end{eqnarray*}}
\newcommand{\beq}{\begin{eqnarray*}}

\newcommand{\D}{\ca{D}}

\newcommand{\pr}{\on{pr}}
\newcommand{\wh}{\widehat}
\newcommand{\wt}{\widetilde}
\newcommand{\mf}{\mathfrak}
\renewcommand{\sl}{\mf{sl}}
\newcommand{\rra}{\rightrightarrows}

\newcommand{\Diff}{\on{Diff}}
\newcommand{\PSL}{\on{PSL}}
\newcommand{\PGL}{\on{PGL}}
\newcommand{\SL}{\on{SL}}
\newcommand{\RP}{\R\!\on{P}}
\renewcommand{\S}{\ca{S}}
\newcommand{\vir}{\mf{vir}}

\newcommand{\CC}{\mathsf{C}}

\newcommand{\ignore}[1]{}

\newcommand{\oz}{\mathsf{o}}
\newcommand{\gz}{\mathsf{g}}
\renewcommand{\subset}{\subseteq}

\newcommand{\DD}{\mathbb{D}}

\renewcommand{\supset}{\supseteq}

\newcommand{\vz}{\mathsf{v}}

\begin{document}
	\sloppy
	\title[Symplectic geometry of projective structures]{Symplectic geometry of projective structures\\ on surfaces with boundary}
	\author{Ahmadreza Khazaeipoul}
	\author{Eckhard Meinrenken}

\begin{abstract}	
For oriented surfaces $\Sigma$ with boundary, we consider the infinite-dimensional deformation space  of projective structures on $\Sigma$ with nondegenerate boundary, up to isotopies fixing the boundary.  We show that this space carries a natural symplectic structure, and is a Hamiltonian space for the symplectic groupoid integrating the 
Adler-Gelfand-Dikii-space of the boundary. 		
	\end{abstract}
	
	\maketitle
	\tableofcontents

\section{Introduction}
According to Adler \cite{adl:tra} and Gelfand-Dikii \cite{gel:fam}, the space $\ca{R}_n(S^1)$ of scalar differential operators on $S^1$
\begin{equation}\label{eq:agd}
 L=\sum_{i=0}^n a_i(t) \f{d^i}{d t^i}
\end{equation}
 with principal symbol $a_n=1$ and subprincipal symbol $a_{n-1}=0$ 
is naturally an infinite-dimensional Poisson manifold. Drinfeld-Sokolov \cite{dri:kdv} gave a construction of this Poisson structure  by Hamiltonian reduction from the  space of $G=\SL(n,\R)$-connections on $S^1$; this correspondence later evolved  into the Beilinson-Drinfeld theory of opers \cite{bei:op}. 

For $n=2$,  the space  $\ca{R}_2(S^1)$ is the space of  \emph{Hill operators}, and the Poisson structure is affine-linear.  It 
is identified with the Kirillov-Kostant-Souriau Poisson structure on the 
dual of the Virasoro Lie  algebra $\mf{vir}(S^1)$, restricted to the  affine subspace at 
level $1$ (see e.g. \cite{kir:orb}). For $n>2$, the 
Adler-Gelfand-Dikii structure is not affine-linear, but one still has an associated Lie algebroid.  An explicit integration to a symplectic groupoid 
\begin{equation} \label{eq:thegroupoid}
\S_n(S^1)\rra \ca{R}_n(S^1)
\end{equation}
was obtained in \cite{al:coad} for $n=2$ and \cite{kha:1} for $n>2$. To describe the underlying groupoid, note that a fundamental system of solutions of $L$ gives a 
nondegenerate quasi-periodic curve 
\begin{equation}
\gamma\colon \R\to S^{n-1},\ \ \ \gamma(t+1)=c\cdot \gamma(t)
\end{equation}
where $c\in G$ is the monodromy of $\gamma$. Conversely, any such curve is realized by some differential operator \eqref{eq:agd}.  
Let $\D_n(S^1)$ denote the space of all nondegenerate quasi-periodic curves
and let $\D_n^{[2]}(S^1)$ denote pairs of such curves that are related by isomonodromic 
deformation. The latter has a pair groupoid structure, 
and the groupoid \eqref{eq:thegroupoid} is realized as its quotient by $G$, 
\begin{equation}
\S_n(S^1)=\D_n^{[2]}(S^1)/G\rra \ca{R}_n(S^1)=\D_n(S^1)/G.
\end{equation}
In \cite{al:symteich}, it was shown that for a  compact oriented surface $\Sigma$ with boundary, 
the infinite-dimensional Teichm\"uller space 
\begin{equation}\label{eq:teich0}
\on{Teich}(\Sigma)=\on{Hyp}(\Sigma)/\on{Diff}_\oz(\Sigma,\p\Sigma)
\end{equation}
is a Hamiltonian Virasoro space at level $1$, with one copy of $\vir(S^1)$ for each boundary circle. Here the quotient is 
by diffeomorphisms fixing the boundary and isotopic to the identity, and the hyperbolic structures are required to have 
\emph{ideal boundary}: they exhibit the same boundary behaviour as for the boundary of  the closed Poincar\'e disk. The moment map 
\[ \Phi\colon \on{Teich}(\Sigma)\to \ca{R}_2(\p\Sigma)\]
takes the class of such a hyperbolic structure to the induced Hill operator on the boundary.  
 
The present article develops a similar picture for the case $n=3$.  Instead of  \eqref{eq:teich0}, we consider 
the deformation space
\begin{equation}
\mf{P}(\Sigma)=\on{Proj}(\Sigma)/\on{Diff}_\oz(\Sigma,\p\Sigma).
\end{equation}
Here $\on{Proj}(\Sigma)$ is the space 
of projective structures on $\Sigma$ with \emph{nondegenerate boundary}, meaning that the restriction of 
any projective chart to the boundary is a nondegenerate curve. By construction, the deformation space comes with a map to the Adler-Gelfand-Dikii space
\begin{equation}\label{eq:psimap}
\Psi\colon \mf{P}(\Sigma)\to \ca{R}_3(\p\Sigma).
\end{equation} 

For the case without boundary $\p\Sigma=\emptyset$, there is a rich and  well-developed theory of the deformation space, and in particular of its subspace $c\mathfrak{P}(\Sigma)$ of \emph{convex} projective structures.  Goldman and Choi \cite{choi:cla} proved that for genus 
$\mathsf{g}\ge 2$, the space 
$c\mathfrak{P}(\Sigma)$ is the Hitchin component of the representation variety of 
$G=\SL(3,\R)$, realizing it as a \emph{higher Teichm\"uller space}  \cite{hit:lie,foc:pro}. 
 In particular,  $\mathfrak{P}(\Sigma)$ is a symplectic manifold of dimension $(2\mathsf{g}-2)\dim G$.  Earlier work of Goldman \cite{gol:con} gave a  Fenchel-Nielsen type parametrization of these projective structures,  leading to explicit Darboux coordinates. The 
full deformation space $\mathfrak{P}(\Sigma)$  is well-understood as well  \cite{choi:con1,choi:con2}; it too has a natural symplectic structure \cite{gol:con,gol:symaff}. 

Our first result extends the construction of a symplectic structure to the case with boundary:  \bigskip

\noindent{\bf Theorem A:} 
For a compact, connected surface with non-empty boundary,  the deformation space $\mathfrak{P}(\Sigma)$ of projective structures with nondegenerate boundary  is naturally an infinite-dimensional symplectic manifold. 
\bigskip

Our second result interprets the boundary restriction \eqref{eq:psimap} as a moment map. 
The statement is particularly satisfying for 
 the deformation space  $c\mf{P}(\Sigma)$ of convex projective structures with positive hyperbolic boundary holonomies. 
 \bigskip

\noindent{\bf Theorem B:}
There is a well-defined action 
of the groupoid $\S_3(\p\Sigma)\rra \ca{R}_3(\p\Sigma)$ on $c\mf{P}(\Sigma)$, 
by variation of the boundary curves. This action is Hamiltonian, with moment 
map given by \eqref{eq:psimap}. 
\bigskip

On the full space $\mf{P}(\Sigma)$, one can consider \emph{small} variations of the  boundary curves. 
The resulting local groupoid action is still compatible with the symplectic structures, but need not extend to a global action. 
\smallskip

The constructions from \cite{al:symteich} where motivated by recent work on JT gravity \cite{mal:con,saa:jt,sta:jt}, and notably by the concept of hyperbolic structures with wiggly boundary. As we shall see, the picture of `wiggly boundaries' is equally prominent in the setting of projective structures. \smallskip

The organization of the paper is as follows. In Section \ref{sec:curves}, we review the theory of nondegenerate quasi-periodic curves,  their relation to the Adler-Gelfand-Dikii space, and Drinfeld-Sokolov reduction. The discussion will be coordinate-free, using abstract circles  rather than $S^1$. The case of $n=3$, with positive hyperbolic monodromy, is described in some detail. Section \ref{sec:projsurface} begins with an overview of the theory of projective structures on compact surfaces without boundary, and extends aspects of this theory to projective structures with nondegenerate boundary. As we shall explain, every convex projective 
structure with nondegenerate boundary and positive hyperbolic monodromy is obtained by attaching half-annuli
to a convex projective structures with \emph{geodesic} boundary; we may hence use Goldman's theory to obtain an explicit Fenchel-Nielsen type description of $c\mf{P}(\Sigma)$. Section \ref{sec:symplectic} contains our construction of a symplectic structure on the space $\mf{P}(\Sigma)$. Similar to \cite{al:symteich}, this symplectic 2-form is obtained by reduction  of an Atiyah-Bott symplectic structure on a suitable space of connections. Section \ref{sec:momentmap} establishes the moment map property. In Section \ref{sec:goldmantwists} we calculate Hamiltonian flows corresponding to  
`Goldman twist'  within our framework.

\bigskip
\noindent{\bf Acknowledgements.} It is a pleasure to thank Anton Alekseev, Boris Khesin, Nathanial Sagman, and Nicolau Cor\c{c}\~ao Saldanha for conversations 
related to this work.

\section{Nondegenerate curves and the Adler-Gelfand-Dikii space} \label{sec:curves}
Throughout, we denote by $\CC$ a compact, connected, oriented 1-manifold -- in other words, a circle but without a preferred choice of coordinate. 
Let $\wt{\CC}$ be a choice of universal cover, so that $\CC=\wt{\CC}/\Z$. The deck transformation corresponding to $1\in \Z$ will be denoted $\kappa\colon\wt{\CC}\to \wt{\CC}$ (moving in the positive direction).
\subsection{Quasi-periodic curves}\label{subsec:quasiperiodic}
Let $G$ be a Lie group and $X$ a homogeneous space for $G$. A curve 
$\gamma\colon 	\wt\CC\to X$
is called \emph{quasi-periodic} if there exists an element $\mu(\gamma)\in G$, called its 
\emph{monodromy}, 
 such that 
\[ \gamma(\kappa(t))=\mu(\gamma)\cdot \gamma(t)\] 
for all $t$. The group $G$ acts on the space of quasi-periodic curves by $(a\cdot\gamma)(t)=a\cdot \gamma(t)$, 
with the monodromy transforming according to  $\mu(a\cdot \gamma)=a\,\mu(\gamma)\, a^{-1}$. 

We shall find it convenient to think of quasi-periodic curves in terms of connections and parallel transport. 
Associated with any quasi-periodic curve is a triple 
\[ (Q,\theta,\tau)\] 
of a principal $G$-bundle $Q\to \CC$ with a connection $\theta$ and a $G$-equivariant map $\tau\colon Q\to X$. 
Here
	\[ Q=(\wt{\CC}\times G)/\sim,\ \ \ \  (t,h)\sim (\kappa(t),\mu(\gamma)h),\]
with connection $\theta$ induced from the trivial connection on $\wt\CC\times G$, and 
\[ \tau\colon Q\to X,\ \  \tau([(t,h)])= h^{-1}\cdot \gamma(t).\] 
Conversely, given such a triple, the choice of a $\theta$-horizontal lift $s\colon\wt\CC\to Q$ gives quasi-periodic paths by composition, $\gamma=\tau\circ s$. As a simple application, we have: 
%

\begin{lemma}[Local sections of monodromy map]\label{lem:deform}
Given a quasi-periodic path $\gamma\colon \wt{\CC}\to X$, there is an open neighborhood $U\subset G$ of $\mu(\gamma)$
and a smooth map $f\colon U\times \wt{\CC}\to X$  such that each $f(h,\cdot)$ is quasi-periodic of monodromy $h$, and 
$f(\mu(\gamma),\cdot)=\gamma$. 
\end{lemma}
\begin{proof}
Let $(Q,\theta,\tau)$ be the triple associated with $\gamma$. Fixing a  base point $t_0\in \wt{\CC}$, we have 
$\gamma=\tau\circ s$ with the unique  horizontal lift  $s\colon \wt{\CC}\to Q$ taking $t_0$ to $[(t_0,e)]$. 
	Choose a smooth family  $\theta_h,\ h\in U$ of connections, with $\theta_g=\theta$, such that the holonomy of $\theta_h$ with respect to $[(t_0,e)]$ is given by $h$. There is a unique $\theta_h$-horizontal lift $s_h\colon \wt{\CC}\to Q$ taking $t_0$ to $[(t_0,e)]$. Then 
	$f(h,\cdot)=\tau\circ s_h$ is the desired family of quasi-periodic paths. 
\end{proof}

\subsection{Nondegenerate curves}
%

We specialize to  $G=\SL(n,\R)$  acting on $X=S^{n-1}$. We shall think of the sphere as a quotient of $\R^n-\{0\}$ by positive scalars, and use 
projective notation $(x_1\colon \cdots \colon x_n)\in S^{n-1}$ for the equivalence class of $(x_1,\ldots,x_n)$.  
A quasi-periodic curve 
\begin{equation}\label{eq:gammacurve}
\gamma\colon 	\wt\CC\to S^{n-1}\end{equation} 
is (positive) \emph{nondegenerate} if it lifts to a curve 
$\ti\gamma\colon \wt\CC\to \R^n-\{0\}$ with the property that 
for any choice of oriented local coordinate $t$ 
the derivatives 
\begin{equation}\label{eq:derivatives}
\ti\gamma(t),\ti\gamma'(t),\ldots,\ti\gamma^{(n-1)}(t)\end{equation}
form an oriented basis of $\R^n$. This condition does not depend on the choice of coordinates or the choice of lift. 
The space of nondegenerate quasi-periodic curves will be denoted $\D_n(\CC)$; the monodromy map 
gives a $G$-equivariant map 
\begin{equation}\label{eq:monodromy}
\mu\colon \D_n(\CC)\to G.
\end{equation}

\begin{example}\label{ex:standardcurves}
	Let $\CC=S^1$, so $\wt{\CC}=\R$. For $c_1<\ldots<c_n$ with $\sum_i c_i=0$, the curve 
	\begin{equation}\label{eq:standardcurve}
	\gamma(t)=(e^{c_1 t}:\cdots:e^{c_n t})
	\end{equation}  	
	is positive nondegenerate, with monodromy the diagonal matrix $\on{diag}(e^{c_1},\cdots,e^{c_n})$.  Indeed, 
	taking $\ti{\gamma}(t)=(e^{c_1 t},\cdots,e^{c_n t})$,  
	the Wronskian determinant is given by $\det(\wt{\gamma},\wt{\gamma}',\ldots)=\prod_{i>j}(c_i-c_j)>0$. 
\end{example}

\begin{remarks}\label{rem:plusminus}
	\begin{enumerate}
		\item 
		We require \eqref{eq:derivatives} to be \emph{oriented} to ensure that the space $\D_n(\CC)$ is connected. 
		\item Let $\on{Fl}^+(\R^n)$ be the manifold of oriented flags in $\R^n$: sequences of oriented subspaces 
$E_i\subset \R^n$ with $\dim E_i=i$ (using the standard orientation for $i=0,n$). By Gram-Schmidt orthogonalization, $\on{Fl}^+(\R^n)\cong \SO(n)$. 
		A nondegenerate curve has a canonical lift 
		\begin{equation}\label{eq:gammalift}
		\begin{tikzcd}
		[column sep={5.5em,between origins},
		row sep={4em,between origins},]
		& \on{Fl}^+(\R^n) \arrow[d]     \\
		\wt{\CC} \arrow[ur, dashed, "\wh{\gamma}"] \arrow[r, "\gamma" '] & S^{n-1}
		\end{tikzcd}
		\end{equation}
		where $\wh{\gamma}(t)$ is the oriented flag whose  $i$-th subspace is spanned 
		by the derivatives \eqref{eq:derivatives} up to order $i-1$. 
	\item By the general discussion from Section \ref{subsec:quasiperiodic}, every $\gamma\in \D_n(\CC)$ determines a 
	triple $(Q,\theta,\tau)$ of a principal $G$-bundle with a connection and equivariant map $\tau\colon Q\to S^{n-1}$. The lift $\wh{\gamma}$ determines a lift of $\tau$ to a $G$-equivariant map 
	\begin{equation}\label{eq:taulift}
	\begin{tikzcd}
	[column sep={5.5em,between origins},
	row sep={4em,between origins},]
	& \on{Fl}^+(\R^n) \arrow[d]     \\
	Q \arrow[ur, dashed, "\wh{\tau}"] \arrow[r, "\tau" '] & S^{n-1}
	\end{tikzcd}
	\end{equation}
	\end{enumerate}
\end{remarks}

Let $Y_n(\CC)=\D_n(\CC)/\sim$
 be the set of components of the fibers of \eqref{eq:monodromy}: that is, paths up to isomonodromic deformations.  
 The quotient map will be denoted 
 \begin{equation}\label{eq:yspace}
 	q\colon  \D_n(\CC)\to Y_n(\CC).
 \end{equation}

 Using Lemma \ref{lem:deform}  to define local charts, we see that 
 $Y_n(\CC)$ is a finite-dimensional manifold, possibly non-Hausdorff, in such a way that 
 the monodromy descends to a local diffeomorphism $Y_n(\CC)\to G$. 

\begin{remarks}
	\begin{enumerate}
		\item For $n=2$, the space $Y_2(\CC)$ is a certain open subset of the universal cover of $\SL(2,\R)$ (see 
		Goldman \cite{gol:the} and Segal \cite{seg:uni}). 
		\item Nondegenerate quasi-periodic paths on spheres are much-studied 
		in the literature.  Little \cite{lit:nondeg} proved that there are exactly three homotopy classes of (positive) 
		nondegenerate \emph{loops} in $S^2$ 
		(see \cite[Section 2.3]{ovs:pro}). 
		The homotopy types of the space of nondegenerate loops on $S^2$ 
		 were computed by Saldanha \cite{sal:hom}. Khesin and B.~Shapiro \cite{khe:non} classified the components of the space of positive, nondegenerate paths $\gamma\colon [0,1]\to S^2$ with prescribed initial and final flags $\wh{\gamma}(0),\wh{\gamma}(1)$.  For spheres of higher dimensions, we mention in particular M.~Shapiro's   generalization of Little's theorem in \cite{sha:top}, and the work of 	  
	  B.~Shapiro and Saldanha \cite{sal:spa} on the homotopy types of spaces of nondegenerate curves.
	\end{enumerate}
\end{remarks}

\subsection{The Adler-Gelfand-Dikii space $\ca{R}_n(\CC)$}
The coordinate-free definition of the Adler-Gelfand-Dikii operators 
\eqref{eq:agd}
involves the space 
 $|\Omega|_\CC^r$ of $r$-densities on $\CC$. We take $\ca{R}_n(\CC)$ to be the space of linear differential operator of order $n$
\begin{equation}\label{eq:L}
	 L\colon |\Omega|_\CC^{\f{1-n}{2}}\to  |\Omega|_\CC^{\f{1+n}{2}},
\end{equation}
with principal symbol 
equal to $1$  and with sub-principal symbol 
equal to $0$. The degrees of the densities are such that these conditions make sense: the principal symbol is a scalar, 
and $L,L^*$ act between the same spaces. The space $\ca{R}_n(\CC)$  is an affine space, with linear differential operators of order $n-2$ as the underling vector space. 

To describe the  relationship with nondegenerate quasi-periodic curves \cite{khe:inf,ovs:sym}, 
we use the n-th order \emph{Wronskian},
\[ W_n\colon |\Omega|_{\wt\CC}^r\times\cdots \times |\Omega|_{\wt\CC}^r\to  |\Omega|_{\wt\CC}^{n(r-\f{1-n}{2})}\]
given in local coordinates  by 
\[  W_n(u_1,\ldots,u_n)=\det\left(\begin{array}{cccc}u_1&\cdots  &u_n\\
	\vdots  &\vdots &\vdots\\
	u_1^{(n-1)}&\cdots &u_n^{(n-1)}
\end{array}
\right).
\]
(The definition does not depend on the choice of coordinates.) Given $L\in \ca{R}_n(\CC)$, 
suppose $u_1,\ldots,u_n\in \Omega^{\f{1-n}{2}}(\wt\CC)$ is a fundamental system of solutions of the lifted operator $\wt{L}$ on $\wt{\CC}$, with Wronskian $W_n(u_1,\ldots,u_n)>0$. (Observe that for $r=\f{1-n}{2}$, the Wronskian is a scalar.)  By Abel's formula from ODE theory, it is in fact constant.  The equivalence class of this $n$-tuple of solutions,  
%
modulo multiplication by positive densities in $\Omega^{\f{1-n}{2}}(\wt\CC)$, 
is a  nondegenerate quasi-periodic curve 
\[ \gamma=(u_1:\ldots:u_n)\colon  \wt\CC\to S^{n-1}.\] 
The curve $\gamma$ is determined by $L$ up to the action of $G$. Conversely, 
every nondegenerate quasi-periodic curve, written as 
$\gamma=(u_1:\cdots:u_n)$, determines $L\in \ca{R}_n(\CC)$ by the formula
\[ L(u)=(-1)^n \f{W_{n+1}(u,u_1,\ldots,u_n)}{W_n(u_1,\ldots,u_n)}.\] 
The resulting map
\begin{equation}\label{eq:p}
p\colon \D_n(\CC)\to \ca{R}_n(\CC)
\end{equation}
is the quotient map for the $G$-action on $\D_n(\CC)$. 

\begin{example} Let $\CC=S^1,\ \wt{\CC}=\R$. For distinct $c_1,\ldots,c_n\in \R$ with $\sum_i c_i=0$, 
	one has the operator 
	\[ L=\prod_{i=1}^n(\f{d}{d t}-c_i).\]
	The functions $u_i(t)=\exp(c_i t)$ are a fundamental system of solutions, defining the quasi-periodic 
	curve $\gamma(t)=(e^{c_1 t}:\ldots:e^{c_n t})$ from Example \ref{ex:standardcurves}.
\end{example}

The infinite-dimensional affine space $\ca{R}_n(\CC)$ carries the \emph{Adler-Gelfand-Dikii Poisson structure} \cite{khe:inf}. According to Khesin and Ovsienko \cite{ovs:sym}, the space of symplectic leaves of $\ca{R}_n(\CC)$ is identified with the space of $G$-orbits in $Y_n(\CC)$. 
This correspondence 
  was enhanced in \cite{al:coad} (for $n=2$) and 
\cite{kha:1} (for $n>2$) to a Morita equivalence of quasi-symplectic groupoids 
\begin{equation}\label{eq:moritaequivalence}
\begin{tikzcd}
[column sep={5.5em,between origins},
row sep={4em,between origins},]
  \ca{S}_n(\CC)   \arrow[d,shift left]\arrow[d,shift right]   & \D_n(\CC) \arrow [dr, "{q}"'] \arrow[dl, "p"]  & G\ltimes Y_n(\CC)   \arrow[d,shift left]\arrow[d,shift right] \\
\ca{R}_n(\CC) & & Y_n(\CC)
\end{tikzcd}
\end{equation}
Here $G\ltimes Y_n(\CC)\rra Y_n(\CC)$ is the action groupoid, with 2-form obtained by pullback of the 2-form on the quasi-symplectic groupoid $G\ltimes G$ integrating the Cartan-Dirac structure on $G$ (see \cite{al:coad}).  
On the other hand, 
$\S_n(\CC)\rra \ca{R}_n(\CC)$ is the quotient
\[ \S_n(\CC)=\D_n^{[2]}(\CC)/G\]
of the  groupoid
\begin{equation} \D^{[2]}_n(\CC)=\D_n(\CC)\times_{Y_n(\CC)} \D_n(\CC)\end{equation}
consisting of pairs $(\gamma_1,\gamma_0)$ of  paths that are homotopic through  
paths in $\D_n(\CC)$ with fixed monodromy. 
The space 
$\D_n(\CC)$ carries a distinguished $G$-invariant  2-form 
$\varpi_\D$, constructed in \cite{al:coad,kha:1}. The difference 
\[ \pr_1^*\varpi_\D-\pr_2^*\varpi_\D\] 
descends to a closed multiplicative 2-form $\omega_\S$ on
$\ca{S}_n(\CC)$, which is shown to be nondegenerate. In this paper, we will regard the symplectic groupoid $(\S_n(\CC),\omega_\S)$ as an alternative description of the Adler-Gelfand-Dikii Poisson structure, avoiding any discussion of Poisson brackets.  	

\begin{remark}\label{rem:inclusion}
There is a canonical affine-linear inclusion, equivariant for the action of $\on{Diff}_+(\CC)$, 
\begin{equation}
\ca{R}_2(\CC)\hra \ca{R}_3(\CC)
\end{equation}	
with image the subspace of operators $L\in \ca{R}_3(\CC)$ with the property $L^*=-L$. It lifts to
an  inclusion $\D_2(\CC)\hra \D_3(\CC)$, equivariant for a Lie group morphism 
$\SL(2,\R)\hra \SL(3,\R)$. In coordinates, the embedding is given by 
\[ \f{d^2}{d t^2}+q(t)\ \ \mapsto\ \ L=\f{d^3}{d t^3}+q(t)\f{d}{d t}+\hh q'(t),\]
and its lift by composition of quasi-periodic curves $\gamma\colon \wt{\CC}\to S^1$  
with the Veronese map $S^1\to S^2,\ (u_1:u_2)\mapsto (u_1^2:u_1u_2:u_2^2)$. For details, see \cite[page 14]{ovs:pro}. 
\end{remark}

\subsection{Drinfeld-Sokolov}
We shall use the  Drinfeld-Sokolov construction to interpret $\ca{R}_n(\CC)$ as a quotient of a certain space of 
principal connections. Our starting point is the choice of a principal $G$-bundle $Q\to \CC$, a $G$-equivariant map 
$\tau\colon Q\to S^{n-1}$,  and a $G$-equivariant lift $\wh{\tau}\colon Q\to \on{Fl}^+(\R^n)$ (cf. \eqref{eq:taulift}). 
Principal connections $\theta\in\A(Q)$ give quasi-periodic curves
by composition $\gamma=\tau\circ s$ where $s\colon \wt{\CC}\to Q$ is a $\theta$-horizontal lift of the quotient map $\wt{\CC}\to \CC$. If this curve $\gamma$ is non-degenerate, it  determines a $G$-equivariant lift (depending on $\theta$)
\begin{equation}\label{eq:whtautheta}
\wh{\tau}_\theta\colon Q\to \on{Fl}^+(\R^n).\end{equation}
Under gauge transformations by $u\in \Gau(Q)$, $\wh{\tau}_{u\cdot \theta}= \wh{\tau}_\theta\circ u^{-1}$.

\begin{definition}\label{def:taupositive}
A connection $\theta\in\A(Q)$ is \emph{$\wh\tau$-positive} if the map $\tau$ restricts to  nondegenerate curves on the horizontal leaves of $\theta$, and the lifted map $ \wh{\tau}_\theta$ is homotopic to $\wh{\tau}$. 
Let $\A^{\wh\tau}(Q)$ be the space of $\wh\tau$-positive connections. 
\end{definition}

Note that this is an open condition: Small perturbations of $\wh\tau$-positive connections are again $\wh\tau$-positive. 
The lift $s\colon \wt{\CC}\to Q$, and hence the curve $\gamma\in \D_n(\CC)$, are defined only up the action of $G=\SL(n,\R)$, but the  resulting operator $L$ does not depend on the choice.  This defines a quotient map 
\begin{equation}\label{eq:gotohill}
\A^{\wh{\tau}}(Q)\to \ca{R}_n(\CC)\end{equation}
Let $\Gau(Q,\tau)$ be the group 
of  gauge transformations $u\colon Q\to Q$ fixing $\tau$, i.e., 
$\tau\circ u^{-1}=\tau$, and let 
$\Gau_\oz(Q,\tau)$ be its identity component. The following is a version of the Drinfeld-Sokolov theorem: 

\begin{proposition}[Drinfeld-Sokolov]\label{prop:ds}
	The action of $\Gau_\oz(Q,\tau)$ on $\A^{\wh{\tau}}(Q)$ is free, with \eqref{eq:gotohill} as its quotient map. 
	This action admits a global section $\ca{R}_n(\CC)\to \A^{\wh{\tau}}(Q)$.
\end{proposition}
The proof involves the Iwasawa decomposition $G=KAN$. Here
$K$ is the rotation group $\SO(n)$, $A$ are the diagonal matrices with positive entries with product equal to $1$, and $N$ are the upper triangular matrices with $1$'s on the diagonal. The group $B=AN$ is the stabilizer of the 
standard flag $\wh{v}_0\in \on{Fl}^+(\R^n)$, with $i$-th subspace spanned by the standard basis vectors $e_1,\ldots,e_i\in \R^n$. The stabilizer $H\subset G$ of $v_0=e_1\in S^{n-1}$  is represented by 
matrices in $\SL(n,\R)$ 
whose first column is a positive multiple of $e_1$. Equivalently, 
\begin{equation}\label{eq:H}
H=(K\cap H) A N,\ \ K\cap H\cong \SO(n-1).
\end{equation}
Note that $\tau\colon Q\to S^{n-1}=G/H$ and its lift $
\wh\tau\colon Q\to \on{Fl}^+(\R^n)=G/B$ are equivalent to 
reductions of structure group of $Q$ to $H$ and $B$, respectively, with $Q_H=\tau^{-1}(eH)$ and $Q_B=\wh{\tau}^{-1}(eB)$.

\begin{proof}[Proof of Proposition \ref{prop:ds}]	
We may assume $\CC=S^1$. The choice of 
a trivialization of $Q_B\subset Q$ identifies $Q\cong S^1\times G$ in such a way that $\wh\tau$ is the map 
\[ \wh\tau\colon S^1\times G\to \on{Fl}^+(\R^n),\ \  (t,h)\mapsto h^{-1}\cdot \wh{v}_0.\]
The groups $\Gau(Q,\tau)\subset \Gau(Q)$ are the loop groups $LH\subset LG$, respectively. 
We will show how to put any given $\theta\in \A^{\wh{\tau}}(Q)$ into a normal form, using a suitable 
transformation in $\Gau_\oz(Q,\tau)$. 

{\bf Claim 1:} There exists a \emph{unique} gauge transformation in $L_\oz(K\cap H)\subset L_\oz H$ taking 
$\wh{\tau}_\theta$ to the `standard lift' $\wh{\tau}$. 
%
Indeed, under the identification $\on{Fl}^+(\R^n)\cong \on{\SO}(n)$, the map  $\wh{\tau}_\theta(\cdot,h)$ is a loop in 
$L_\oz \SO(n)=L_\oz K$, which actually is contained in $L_\oz (K\cap H)=L_\oz (\SO(n-1))$ (since it is a lift of $\tau$). The desired first gauge transformation is just the inverse of this loop.

Having thus arranged that 
$\theta$ is a $\wh\tau$-positive connection with the additional property $\wh{\tau}_\theta=\wh{\tau}$, we shall prove:

{\bf Claim 2:}  The  $\wh\tau$-positive connections with  $\wh{\tau}_\theta=\wh{\tau}$
are given by connection 1-forms $\mathsf{A}=\xi \d t$ with 
\begin{equation}\label{eq:ds-step1} \xi=\left(
\begin{array}{ccccc}
*&\cdots&\cdots&*&*\\
r_1&\cdots&\cdots&*&*\\
0&r_2&\cdots&*&*\\
\cdots&\cdots&\cdots&\cdots&\cdots\\
0&0&\cdots&*&*\\
0&0&\cdots&r_{n-1}&*
\end{array}
\right)\end{equation}
where $r_1,\ldots,r_{n-1}>0$. 
To prove this, let $\mathsf{A}=\xi \d t$ be the connection 1-form corresponding to $\theta$. A horizontal lift is of the form $s(t)=(t,h(t))$ where $h\colon \R\to G$ is a solution of 
\[ h'(t)h(t)^{-1}=-\xi(t).\] 
The corresponding curve in $\R^n-\{0\}$ is $\ti\gamma(t)=h(t)^{-1}e_1$. (Here we are writing the curve as \emph{column vectors}.) 
Our assumptions (nondegeneracy, and $\wh{\tau}_\theta=\wh{\tau}$) 
mean that the vectors
\[ v_{i}(t)=h(t)\ti{\gamma}^{(i-1)},\ \  i=1,\ldots,n\]
are linearly independent, and 
generate the standard oriented flag. That is, $v_i(t)$ lies in the span of $e_1,\ldots,e_i$, and the coefficient of $e_i$ is strictly 
positive. These vectors are recursively given by $v_1(t)=e_1$, and 
\[ v_{i+1}(t)=v_i'(t)+\xi(t)v_i(t).\]
It follows that $\xi(t)v_i(t)$ lies in the span of $e_1,\ldots,e_{i+1}$, with a strictly negative coefficient in front of 
$e_{i+1}$. By another induction, we hence see that $\xi(t)e_i$ (which is the $i$-th column of the matrix $\xi(t)$) has the same property. 
This proves Claim 2.

After this preparation, the 
standard Drinfeld-Sokolov theory  \cite{bei:op,khe:inf} gives is a \emph{unique} gauge transformation in $LB$ putting $\xi$ into the normal form 
\begin{equation}\label{eq:ds} \left(
\begin{array}{ccccc}
0&\cdots&\cdots&0&-a_0\\
1&\cdots&\cdots&0&-a_1\\
0&1&\cdots&0&-a_2\\
\cdots&\cdots&\cdots&\cdots&\cdots\\
0&0&\cdots&0&-a_{n-2}\\
0&0&\cdots&1&0
\end{array}
\right).\end{equation}
Here $a_0,\ldots,a_{n-2}$ are the coefficients of the operator $L$ corresponding to $\theta$. 
The set of matrices of this form defines an affine slice for the action. 
\end{proof}

\begin{remark}
	Claim 2 shows that the regular connections with 
	$\wh{\tau}_\theta=\wh{\tau}$ are exactly the Beilinson-Drinfeld oper connections \cite{bei:op} for the reduction of structure group $Q_B\subset Q$. 	
\end{remark}

\begin{remark}\label{rem:segal}
The construction of the slice in this proposition made use of coordinates. There is a \emph{distinguished} choice of the bundle $Q$, with reduction of structure group $Q_B$, for which the  \emph{Drinfeld-Sokolov slice} 
\begin{equation} \ca{R}_n(\CC)\hra \A^{\wh{\tau}}(Q)\end{equation} 
is completely canonical.  See the discussion in Segal's paper \cite{seg:geo} (cf.~ \cite{al:coad}). In a nutshell,  every  $L\in \ca{R}_n(\CC)$ defines a linear connection on the $(n-1)$-jet bundle $E=J^{n-1}(|\Lambda|_\CC^{\f{1-n}{2}})$
of the line bundle  of $\f{1-n}{2}$-densities. 
This bundle $E$ has a natural volume form (defined by the Wronskian), and a 
filtration given by the kernels of the maps from $(n-1)$-jets to $(n-k)$-jets. One takes 
$Q$ to be the associated principal $G$-bundle, with reduction of structure group defined by the filtration. 
\end{remark}

Recall that the space $\A(Q)$ of principal connections  is an affine space, equipped with an affine action of the gauge group $\Gau(Q)$. The choice of an invariant metric on $\g=\mf{\sl}(n,\R)$, which we shall take to be the trace form, $\xi_1\cdot \xi_2=\on{tr}(\xi_1\xi_2)$,
makes $\A(Q)$ into an infinite-dimensional Poisson manifold, and the action of $\Gau(Q)$ is Hamiltonian, with affine moment map the identity map of $\A(Q)$. The (affine) moment map for the action of the subgroup 
$\Gau_\oz(Q,\tau)$ is the  projection 
\begin{equation}\label{eq:aqtau}
 \A(Q,\tau)=\A(Q)/\on{ann}(\gau(Q,\tau)).
\end{equation}
\begin{proposition}\label{prop:singleorbit}
The image of $\A^{\wh{\tau}}(Q)$ under the moment map \eqref{eq:aqtau}
for the $\Gau_\oz(Q,\tau)$-action 
is a single $\Gau_\oz(Q,\tau)$-orbit 
\[ \O\subset \A(Q,\tau).\] 
The $\Gau_\oz(Q,\tau)$-action on this orbit is free. 
\end{proposition}	
\begin{proof}
Using the same trivialization as in the proof of Proposition \ref{prop:ds}, we have that 
$\on{ann}(\gau(Q,\tau))$ is the space of 1-forms on $S^1$ with values in $\on{ann}(\h)=\h^\perp$. These are the matrices whose nonzero entries are in the first row. 
In the proof of \ref{prop:ds}, we used that every matrix of a $\wh\tau$-regular connection can be put into the form 
\eqref{eq:ds-step1} by a gauge transformation in $L(K\cap \SO(2))$, 
which, in turn,  can be put into 
Drinfeld-Sokolov normal form \eqref{eq:ds} by a gauge transformation in $LB$, with all non-constant  entries in the last column. But we may also use a gauge transformation in $LB$ to put \eqref{eq:ds-step1} into a normal form $\xi_1\ d t$ where all such entries appear in the first row, 
\begin{equation}\label{eq:ds1} \xi_1=\left(
	\begin{array}{ccccc}
		0&*&\cdots&*&*\\
		1&\cdots&\cdots&0&0\\
		\cdots&\cdots&\cdots&\cdots&\cdots\\
		0&0&\cdots&0&0\\
		0&0&\cdots&1&0
	\end{array}
	\right)\end{equation} 	
The quotient by $\on{ann}(\gau(Q,\tau))$ amounts to considering connection 1-forms $\xi$ 
modulo their first row. The slice \eqref{eq:ds1}, under this quotient map, becomes a  single point. 
 \end{proof}
In summary, this presents the Poisson space $\ca{R}_n(\CC)$ as a Hamiltonian reduction at level $\O$, (Drinfeld-Sokolov reduction) of the space of connections by the action of $\Gau_\oz(Q,\tau)$:
\begin{equation}
\ca{R}_n(\CC)=\A(Q)_\O.
\end{equation}


\subsection{Positive hyperbolic elements}\label{subsec:positivehyper}
We shall now focus on the case $n=3$. An element  $g\in \SL(3,\R)$ is called \emph{positive hyperbolic} if it has three distinct, positive eigenvalues. Order these eigenvalues as  
\begin{equation}\label{eq:eigenvalues}
e^{c_1}<e^{c_2}<e^{c_3},\ \  \sum_i c_i=0.
\end{equation}
 Let $t\mapsto g^t$ denote the unique 1-parameter subgroup through $g^1=g$; it defines a flow $\vz\mapsto g^t\cdot\vz$ on $S^2$. There is a triangulation of $S^2$, invariant under the flow, 
with six vertices, twelve edges, and eight triangles. The open triangles are  the 
images of open cones in $\R^3$ spanned by triples of eigenvectors  for the eigenvalues. A triangle will be called \emph{positive} if this basis of eigenvectors is oriented, and negative otherwise. 

The three vertices of a triangle, given by images of eigenvectors corresponding to \eqref{eq:eigenvalues}, are respectively a repeller, saddle, and attractor for the flow. 

\begin{proposition}
	Suppose $\gamma\in \D_3(\CC)$ has positive hyperbolic monodromy $g$, and that $\gamma$ is (globally) convex: it intersects any  projective line in at most two points.  Then 
	\begin{enumerate}
		\item The curve $\gamma$ is contained a unique positive triangle $\Delta^\gamma\subset S^2$ (for the triangulation defined by $g$).
		\item The limits $\gamma(\pm\infty)=\lim_{t\to \pm \infty}\gamma(t)$ exist, and are 
		vertices of $\Delta^\gamma$. 
		\item The curve $\gamma$ admits an isomonodromic deformation  to the curve $t\mapsto g^t \cdot v$, for any choice of $\vz\in S^2$ and any choice of parametrization $\CC\cong S^1$. (That is, the two curves are 
		in the same fiber of $q\colon \D_3(\CC)\to Y_3(\CC)$.) 
	\end{enumerate}
\end{proposition}

\begin{center}
	\includegraphics[scale=0.3]{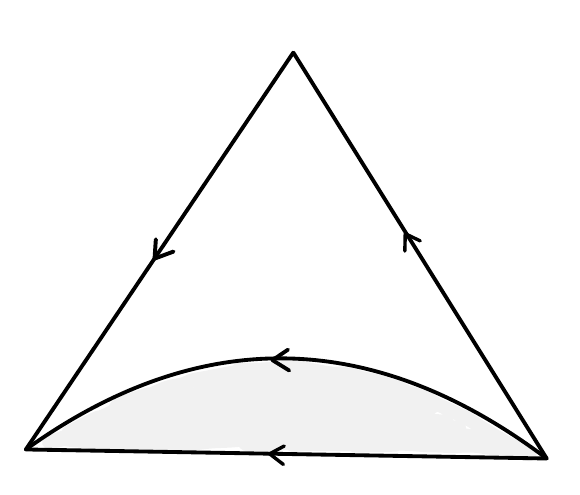}
\end{center}
\begin{proof}
	We may assume $\CC=S^1,\ \wt\CC=\R$. 	Using equivariance  with respect to the action of $\SL(3,\R)$, we may 
	arrange that 
	\begin{equation}\label{eq:positivehyperbolic}
	g=\on{diag}(e^{c_1},e^{c_2},e^{c_3}),
	\end{equation}
	 so the triangulation of $S^2$ is the standard one (given by images of standard orthants in $\R^3$). 
	If $\gamma$ were to meet the great circle line defined by  $x_i=0$,  it would have to cross that great circle an infinite number of times, by quasi-periodicity. For a convex curve, this is impossible. It follows that  every convex curve is fully contained in a fixed component of the set where $x_1x_2x_3\neq 0$. Using the action of a diagonal element of $\SL(3,\R)$, we may arrange that the curve  is contained in the region where $x_1>0,x_3>0$, and $x_2\neq 0$. 
	
	Introduce coordinates $u=x_2/x_1,\ v=x_3/x_1$ with $v>0$. The flow $g^t$ is given by 
	\[ (u,v)\mapsto (e^{t(c_2-c_1)}u,e^{t(c_3-c_1)}v).\]
	Its trajectories  are the curves $v=c\,u^b$ with $b=\f{c_3-c_1}{c_2-c_1}>0$ and $c\in \R$. 
	Looking at the sequence of points $\gamma(n)=g^n\gamma(0),\ n\in \Z$, we see that 
	the initial point $\gamma(0)$ (and hence the entire curve $\gamma$) must lie 
	in the region where $u>0$ (i.e., $x_2>0$), and its graph is of the form $v=f(u)$ where $f'>0$ and $f''>0$. 
	Given two convex 
	nondegenerate curves $\gamma_0,\gamma_1$ with monodromy $g$, both contained in this region, 
	their linear interpolation $\gamma_s(t)=(1-s)\gamma_0(t)+s\gamma_1$ (with sums taken in $\R^2$, for the choice of coordinates $u,v$)  again has this property. In particular, $\gamma(t)$ is homotopic to $t\mapsto g^t\cdot v$, for any choice of $v\in \Delta^\gamma$. The claim about limits follows since $\lim_{t\to \pm\infty}\gamma(t)=\lim_{n\to \pm \infty}\gamma(n)
	=\lim_{n\to \pm \infty} g^n\gamma(0)$. 
\end{proof}

The unique convex geodesic segment $s^\gamma\subset S^2$ connecting the limit points $\gamma(\pm\infty)$ is one of the open edges of  the closed triangle $\ol{\Delta}^\gamma$. It is called the \emph{principal segment} in Goldman's terminology \cite{gol:con}. Later, we will also be interested in the convex region 
\[ \Omega^\gamma\subset S^2\]
(sketched as the shaded region in the picture above), given by all points in the open region between $\gamma$ and the principal segment $s^\gamma$, together with those two curves
(but excluding $\gamma(\pm \infty)$).  

\begin{remark} 
For a given positive hyperbolic element $g$, the convex curves of monodromy $g$ can lie in any of the four positive 
triangles. This accounts for \emph{four} components of $\mu^{-1}(g)$. 
	In addition, $\mu^{-1}(g)$ also contains non-convex curves. In fact, \emph{any} quasi-periodic curve in $S^2$ with monodromy $g$ may be $C^0$-approximated by  nondegenerate curves, using the `telephone wire' construction  \cite{sal:hom}.  
	By the results of Khesin-Shapiro \cite{khe:non}, these non-convex curves give rise to two additional components  of $\mu^{-1}(g)$. 
\end{remark}

We shall reserve the notation $\Delta$ for the standard open triangle 
\[ \Delta=\{(x_1\colon x_2\colon x_2)\in S^2|\ x_1>0,x_2>0,x_3>0\}.\]
It is invariant under the action of the subgroup $A\subset \SL(3,\R)$ of diagonal matrices with positive entries. 
Let $A_+\subset A$ be the subset 
of matrices of the form $g=\on{diag}(e^{c_1},e^{c_2},e^{c_3})$ with $c_1<c_2<c_3$ and $\sum c_i=0$, and denote by 
\begin{equation}\label{eq:subdelta}
\D_3(\CC)_\Delta\subset \D_3(\CC)
\end{equation}
the space of all \emph{convex} curves in $\D_3(\CC)$, with  monodromy in $A_+$ 
and with range in $\Delta\subset S^2$. Given $\gamma\in \D_3(\CC)_\Delta$, written as $\gamma(t)=(1\colon u(t)\colon v(t))$, 
we may use $\f{1}{c_3-c_1}\log(v)$ as a coordinate on $\wt\CC$,  identifying $\CC=S^1$. In these coordinates, we have that 
\begin{equation}\label{eq:varphicurve} \gamma(t)=(e^{tc_1}\colon \delta(t) e^{tc_2}\colon e^{tc_3})\end{equation}
where $\delta\in C^\infty(\R)$ is positive and \emph{periodic}. By calculating the Wronskian, we see that convexity 
of $\gamma$ is equivalent to the inequality
\[  \delta''+3c_2 \delta'-(c_2-c_1)(c_3-c_2)\delta<0.\]
We note that the solutions of this inequality form a convex cone, containing  $\delta=1$. 

\begin{remark}
For a globally convex quasi-periodic curve $\wt{\CC}\to S^2$, the monodromy is either positive hyperbolic (the case considered above), or is conjugate to a $3\times 3$ Jordan block with eigenvalues $1$, or to the direct sum of a $2\times 2$ Jordan block with eigenvalue $\lambda>0$ and a $1\times 1$ block with entry $\lambda^{-2}$. See \cite[Section 2.2.8]{cas:mod} for a related discussion. 
\end{remark}

\section{Projective structures on surfaces}\label{sec:projsurface}
\subsection{$(G,X)$-structures}\label{subsec:gxstructures}
Projective structures may be considered within the framework of \emph{geometric structures} \cite{gol:geo}. Let  $X$ be a homogeneous space for a Lie group $G$. A $(G,X)$-structure on a connected manifold $M$ is described by an atlas with $X$-valued charts, with transition functions given by elements of $G$. 
These data determine a triple 
\begin{equation}\label{eq:triple}
(P,\theta,\sigma)\end{equation} 
of a principal $G$-bundle $P\to M$ with a flat connection $\theta$ and a
$G$-equivariant map $\sigma\colon P\to X$
satisfying 
\begin{equation}\label{eq:transversality}
\ker(T\sigma)\cap \ker\theta=0.
\end{equation}
Here $(P,\theta)$ is the flat bundle defined 
by the (constant) transition functions. Equivalently, the fibers $P|_m$ are the germs $[\phi]_m$ of charts 
$\phi\colon U\to X$ around $m$; the flat connection is such that for every chart $(U,\phi)$ the resulting map 
$U\to P|_U,\ m\mapsto [\phi]_m$ is horizontal.  The map $\sigma$ is given by 
\begin{equation}\label{eq:sigma}
 \sigma\colon P\to X,\ [\phi]_m\mapsto \phi(m).
 \end{equation}
 It may be regarded as a section of the associated bundle $P\times_GX$, and is called the 
 \emph{developing section}. Given a simply connected covering space $\wt{M}$ with group $\Gamma$ of deck transformations, 
 one obtains a \emph{developing map} $\varphi=\sigma\circ s\colon \wt{M}\to X$ for the $(G,X)$-structure, 
 by composing 
  $\sigma$ with a $\theta$-horizontal lift of the quotient map, 
 \[ 
 \begin{tikzcd}
 [column sep={5.5em,between origins},
 row sep={4em,between origins},]
 & P \arrow [r, "\sigma"'] \arrow[d] & X\\
 \wt{M} \arrow[ur, dashed, "s"] \arrow[r, "/\Gamma" '] & M & 
 \end{tikzcd}
 \]
 The developing map $\varphi$ is equivariant for a unique homomorphism $\varrho\in \Hom(\Gamma,G)$. The choice of $s$, and hence the map $\varphi$,  is unique up to the action of $G$; a different choice changes $\varrho$ by conjugation. 
 
 In the opposite direction, any triple \eqref{eq:triple} of a principal $G$-bundle with a flat connection and a
$G$-equivariant map $\sigma\colon P\to X$ satisfying \eqref{eq:transversality} defines a $(G,X)$-structure on $M$, 
with developing map $\varphi=\sigma\circ s$. 
The triple defined by this $(G,X)$-structure is canonically isomorphic to $(P,\theta,\sigma)$. 
\begin{remark}
	In our discussion below, we will usually \emph{fix} the pair $(P,\sigma)$. Writing $X=G/H$, the choice of $\sigma$ is equivalent to a reduction of structure group $P_H\subset P$ to $H\subset G$. (Hence, the choice of $(P,\sigma)$ is equivalent to a choice of principal $H$-bundle.) A connection on $P$ satisfying the transversality condition \eqref{eq:transversality} is equivalent to a 
	Cartan connection \cite{kob:tr}  on  $P_H$. 
\end{remark}

\begin{remark}
Condition \eqref{eq:transversality} means that the bundle map 
\begin{equation}
T\sigma|_{\ker\theta}\colon \ker\theta\to TX
\end{equation}
is a fiberwise isomorphism. If $M$ and the model space $X$ are oriented, then $\ker\theta$ and $TX$ inherit orientations, and we may require that this map is fiberwise 
\emph{orientation preserving}.  
\end{remark}

\subsection{Projective structures}
A \emph{projective structure}  on a surface $\Sigma$ is a $(\on{PGL}(3,\R),\RP^2)$-structure. If $\Sigma$ is oriented, then any 
chart map admits a unique lift to an orientation preserving map to the sphere $S^2$, with transition maps in $\SL(3,\R)$. 
Since orientations will be relevant for our discussion, we will therefore treat projective structures  on oriented surfaces 
as  $ (G,X)=(\SL(3,\R),S^2)$-structures. 
 
Projective structures may be described by their (orientation preserving) developing maps $\varphi\colon \wt\Sigma\to S^2$, 
equivariant with respect to a morphism 
$\varrho\in \Hom(\Gamma,G)$; the pair $(\varphi,\varrho)$ is unique up to the action of $G$. 
A \emph{geodesic} on $\Sigma$ is a curve that, locally in a chart, is given by a geodesic in $S^2$. 
The projective structure  is called \emph{convex} if every path $[0,1]\to \Sigma$ is homotopic relative fixed end points to a geodesic, 
uniquely up to reparametrization.  Equivalently \cite{gol:con}, the developing map $\varphi\colon \wt{\Sigma}\to S^2$ is injective, with convex image contained in an open hemisphere.

If $\Sigma$ has no boundary, we denote by $\on{Proj}(\Sigma)$ the space of all projective structures, and by  $\on{cProj}(\Sigma)$ the subspace of convex projective structures. 
The group $\on{Diff}_\oz(\Sigma)$ of diffeomorphisms isotopic to the identity acts on these spaces, and the quotient spaces are the \emph{deformation spaces}, 
\[ c\mf{P}(\Sigma)=\frac{\on{cProj}(\Sigma)}{\on{Diff}_\oz(\Sigma)}\subset \mathfrak{P}(\Sigma)=\frac{\on{Proj}(\Sigma)}{\on{Diff}_\oz(\Sigma)}.\]
%

We list some basic examples of projective structures, following \cite{choi:cla,gol:con}. 

\begin{example}[Hyperbolic structures]  Let $\DD$ denote the Poincar\'{e} disk, with its usual action of $\on{PSU}(1,1)$. 
A hyperbolic structure on $\Sigma$ is a $(\on{PSU}(1,1),\DD)$-structure; let $\on{Hyp}(\Sigma)$ be the space of hyperbolic structures. 
The Klein-Beltrami model of hyperbolic geometry defines an inclusion 
\begin{equation}\label{eq:klein2} 
\DD\hra S^2
\end{equation}
taking the center $0\in\DD$ to $(1\colon 0\colon 0)\in S^2$, and equivariant for a group morphism
\begin{equation}\label{eq:klein1} \on{PSU}(1,1)\hra \SL(3,\R).\end{equation}
(We may define \eqref{eq:klein1} via the adjoint representation of $\on{PSU}(1,1)$ on its Lie algebra
$\mf{psu}(1,1)\cong \R^3$.) 
Hence, any atlas with hyperbolic charts may be regarded as an atlas with projective charts.

If $\Sigma$ is compact, of negative Euler characteristic, then the projective structures arising in this way are convex, since $\wt\Sigma\cong \DD$ by uniformization. 
This defines an inclusion $\on{Hyp}(\Sigma)\to c\on{Proj}(\Sigma)$. 
Passing to isotopy classes, we obtain an inclusion 
\begin{equation}\label{eq:teich}
 \on{Teich}(\Sigma)\to c\mf{P}(\Sigma)\end{equation}
of Teichm\"uller space.  
\end{example}

\begin{example}[Principal annulus \cite{gol:con}]  \label{ex:principalannulus}
Let $g=\on{diag}(e^{c_1},e^{c_2},e^{c_3})\in A_+$, i.e., $c_1<c_2<c_3$ with $\sum c_i=0$. The map 
%
\begin{equation}\label{eq:principalannus}
\varphi\colon \R^2\to S^2,\ (x,y)\mapsto (e^{xc_1}\colon -y e^{xc_2}\colon e^{xc_3})
\end{equation}
is an embedding onto the convex subset $\Omega$ of all $(x_1\colon x_2\colon x_3)$ such that $x_1,x_3>0$. It intertwines the transformation 
$(x,y)\mapsto (x+1,y)$ with the action of $g$. Hence, $\varphi$ serves as a developing map for 
a convex projective structure on the open annulus $\R/\Z\times\R$. This is the  
 \emph{principal annulus of monodromy $g$}. The curve $y=0$ is the unique closed geodesic  for this projective structure. 
It separates the principal annulus into two \emph{half-annuli}, defined by $y\ge 0$ and $y\le 0$.  
\end{example}

\begin{example}[Surfaces of negative Euler characteristic]\label{ex:goldman}
The work of Choi and Goldman, building on earlier work of Kuiper,  Benz\'ecri, Koszul, Kac-Vinberg, and others, gave a complete classification of projective structures on compact oriented surfaces $\Sigma$ without boundary and with negative Euler characteristic, $\chi(\Sigma)<0$.
The main result of 
\cite{choi:con} states that the deformation space of \emph{convex} projective structures is the Hitchin component 
(higher Teichm\"uller space) \cite{hit:lie}
of the character 
variety of $\SL(3,\R)$.  Furthermore, by Goldman's work \cite{gol:con} (see also Kim \cite{kim:sym}) this space admits \emph{Fenchel-Nielsen parameters}: 
\begin{itemize}
	\item  The holonomies along simple loops in $\Sigma$ are 
	positive hyperbolic elements; the conjugacy class of the holonomy assigns two \emph{length parameters}	to the loop. 
	\item Every simple loop is freely homotopic to a unique closed geodesic. 
	The local model around this geodesic is given by the principal annulus (Example \ref{ex:principalannulus}). 
	\item Cutting along non-intersecting closed geodesics, one obtains a pants decomposition of the surface.    
	The projective structure on each 
	pair of pants with geodesic boundary is determined by eight parameters:
	two length parameters for each boundary component, and two \emph{internal parameters}. 
	\item For each gluing circle there are also two \emph{twist parameters}, since pants may be glued with a twist, similar to the case of hyperbolic structures. (See Section \ref{subsec:twists} below.) 
	\item For a surface of genus $\gz\ge 2$, there are $2\gz-2$ pants with $3\gz-3$ gluing circles. The convex projective structures are thus described by $(2\gz-2)\cdot 2+(3\gz-3)\cdot (2+2)=(2\gz-2)\cdot \dim G$ parameters. 
\end{itemize} 
More general projective structures, not necessarily convex, are obtained from the convex ones by 
gluing-in of $\pi$-annuli
\cite{choi:con1,choi:con2,choi:cla}.
\end{example}

\subsection{Goldman twists}\label{subsec:twists}
Consider again the principal annulus, Example \ref{ex:principalannulus}. The action of $A\subset G$ on $S^2$ preserves $\Omega$ and commutes with the action of $g$, and so descends to an action  on the principal annulus by projective transformations.  
 Given $h\in A$, we may choose an orientation preserving diffeomorphism $\mathsf{F}$ of the annulus that is trivial for $y<-\epsilon$ and agrees with the action of $h$  for $y>\epsilon$, for 
some $\epsilon>0$. Push-forward under $\mathsf{F}$ produces a new projective structure on the annulus, which agrees with the original one for $|y|>\epsilon$, but is deformed inside the region $|y|\le \epsilon$. The diffeomorphism $\mathsf{F}$, and hence the resulting projective structure, is unique up to a diffeomorphism that is supported inside the $\epsilon$-neighborhood, 
and (in particular) is isotopic to the identity. 

Given a projective structure on a general surface $\Sigma$, as in Example \ref{ex:goldman}, and a simple closed oriented geodesic $\lambda\subset \Sigma$ with positive hyperbolic holonomy, 
we may use the principal annulus as a model
to `twist' the projective structure by elements $h\in A$. Up to isotopy, the new projective structure depends only on $\lambda$ and 
$h$.  We hence obtain, for any $\lambda$, an action of  $A$ on the space $c\mathfrak{P}(\Sigma)$, called the \emph{Goldman 
twist}. The Goldman twists for distinct simple closed oriented geodesics commute.
 
Taking $\lambda$ to be one of the curves from a pants decomposition, these twists change the twist parameters corresponding to $\lambda$ but do not affect any of the other parameters. 
In \cite{wie:def}, Wienhard and Zhang also introduced \emph{eruption flows},  changing the internal parameters of a pair of pants
(but not affecting the twist or length parameters).

\subsection{Projective structures with nondegenerate boundary}
The Choi-Goldman theory also applies to projective structures on surfaces with boundary, provided that the boundary components are \emph{geodesics} with positive hyperbolic holonomy.  In this article, we will consider the opposite extreme: 

\begin{definition}
Let $\Sigma$ be a connected oriented surface with boundary. A projective structure on $\Sigma$, with developing map $\varphi\colon \wt{\Sigma}\to S^2$,  has \emph{ nondegenerate boundary}, if the restriction of $\varphi$ to boundary components
$\wt\CC\subset \p\wt{\Sigma}$ are nondegenerate curves.  
We denote by $\on{Proj}(\Sigma)$ the space of projective structures with  nondegenerate boundary, and 
by $\on{cProj}(\Sigma)$ the subspace for which (i) the projective structure is convex, and (ii) all boundary holonomies are 
positive hyperbolic. The corresponding deformation spaces are denoted  
\[ c\mathfrak{P}(\Sigma)=\f{c\on{Proj}(\Sigma)}{\Diff_\oz(\Sigma,\p\Sigma)},\ \ \ 
\mathfrak{P}(\Sigma)=\f{\on{Proj}(\Sigma)}{\Diff_\oz(\Sigma,\p\Sigma)},\]
where $\Diff(\Sigma,\p\Sigma)$ is the group of diffeomorphisms fixing the boundary (pointwise), and $\Diff_\oz(\Sigma,\p\Sigma)$ 
its identity component.
\end{definition}
Convexity means, as before, that every path $[0,1]\to \Sigma$ is homotopic (rel fixed end points) to a unique geodesic, or equivalently that the image of a developing map $\varphi$ is convex.  If $\p\Sigma\neq\emptyset$, the spaces $c\mathfrak{P}(\Sigma)\subset \mathfrak{P}(\Sigma)$ are infinite-dimensional. 

Consider a  boundary component $\wt{\CC}\subset\p \wt{\Sigma}$ mapping to $\CC\subset \p\Sigma$. The nondegenerate curve 
$\gamma\colon \wt{\CC}\to S^2$ determines an Adler-Gelfand-Dikii operator  $L\in \ca{R}_3(\CC)$. Changing the pre-image of $\CC$ in $\p \wt{\Sigma}$, or changing $\varphi$ by the $G$-action, does not affect this operator. Additionally,  $L$ is 
unchanged under the diffeomorphisms fixing the boundary. Using this construction for all components of $\p\Sigma$, 
we obtain a map, 
\begin{equation}\label{eq:boundaryrestriction} 
\Psi\colon \mathfrak{P}(\Sigma)\to \ca{R}_3(\p\Sigma),
\end{equation}
equivariant for the residual action of $\Diff_+(\Sigma)/ \Diff_\oz(\Sigma,\p\Sigma)$.

\subsection{Examples}
\begin{example}[Hyperbolic 0-metrics]\label{ex:hyperbolic}
	In \cite{al:symteich}, a hyperbolic structure with ideal boundary 
	is defined in terms of an atlas with charts taking values in the \emph{closed} Poincar\'{e} disk $\ol{\DD}$, with transition functions given by elements of $\on{PSU}(1,1)$.  (This is not literally a $(G,X)$-structure, since $\ol{\DD}$ is not a homogeneous space.) The embedding \eqref{eq:klein2} given by the Klein-Beltrami model extends to an embedding $\ol{\DD}\to S^2$ of the closed Poincar\'{e} disk, 
	valued in the hemisphere where $x_1>0$.  We hence obtain a map $\on{Hyp}(\Sigma)\to \on{cProj}(\Sigma)$, 
	intertwining the map $\ca{R}_2(\p\Sigma)\to \ca{R}_3(\p\Sigma)$ (cf.~ Remark \ref{rem:inclusion}) of the boundary restrictions. Passing to deformation spaces, 
	we arrive at the diagram 
	\begin{equation}\label{embedding}
	\begin{tikzcd}
	[column sep={6.5em,between origins},
	row sep={4em,between origins},]
	 \on{Teich}(\Sigma)  \arrow[d,hook']\arrow[r,"\Phi"]   &   \ca{R}_2(\p\Sigma)  \arrow[d,hook'] \\
	\mf{P}(\Sigma)\arrow [r,"\Psi"]  & \ca{R}_3(\p\Sigma)
	\end{tikzcd}
\end{equation}
\end{example}

\begin{example}[Disk]
As a special case of Little's theorem \cite{lit:nondeg}, any two (globally) \emph{convex}  nondegenerate loops $\lambda\colon S^1\to S^2$ are homotopic. 
Extending $\lambda$ to an embedding of the disk $D^2$ (such an extension is unique up to isotopy fixing the boundary), we obtain a convex projective structure on the disk, with  nondegenerate boundary. In this way, the space of convex projective structures with  nondegenerate boundary is identified with the space of convex  nondegenerate loops $\lambda\colon S^1\to S^2$, up to the action of $G$. By comparison, the Teichm\"uller space of hyperbolic $0$-metrics on the disk is a quotient 
$\on{Diff}_+(S^1)/G$ of orientation-preserving diffeomorphisms modulo the projective ones \cite{al:symteich}. 

As a simple example of a \emph{non-convex} projective structure on the disk, take an embedding $D^2\to S^2$ whose image contains a hemisphere, and with boundary a convex curve. Its composition with the quotient map gives a developing map $D^2\to S^2$ that is not injective. 
\end{example}

\begin{example}[Annulus] \label{ex:annulus}
Recall from Example \ref{ex:principalannulus} the construction of a projective structure on the open annulus $\R/\Z\times \R$, with monodromy $g=\on{diag}(e^{c_1},e^{c_2},e^{c_3})\in A_+$. Its restriction to 
$\R/\Z\times [-1,1]$ defines a projective structure on the \emph{closed} annulus, 
with nondegenerate boundary. More generally, given two curves $\gamma_\pm\in\D_3(S^1)_\Delta$, written in the form 
$\gamma(t)=(e^{t c_1}\colon \delta_\pm (t) e^{t c_2}\colon e^{tc_3})$ with periodic functions $\delta_\pm$ 
as in \eqref{eq:varphicurve}, the subset where  
$-\delta_+(x)\le y\le \delta_-(x)$  is a closed annulus, which may be identified with $\R/\Z\times [-1,1]$ after  reparametrization.

For a coordinate-free  description of this example, let $\CC$ be an `abstract' circle, 
and suppose  $\gamma\in\D_3(\CC)$ has positive hyperbolic monodromy $g$. Let $\Omega^\gamma\subset S^2$ be the convex area enclosed by the curve $\gamma$ and the geodesic arc $s^\gamma$ between $p_-=\gamma(-\infty)$ and  $p_+=\gamma(\infty)$, 
including both curves but excluding $p_\pm$: 
\begin{center}
	\includegraphics[scale=0.2]{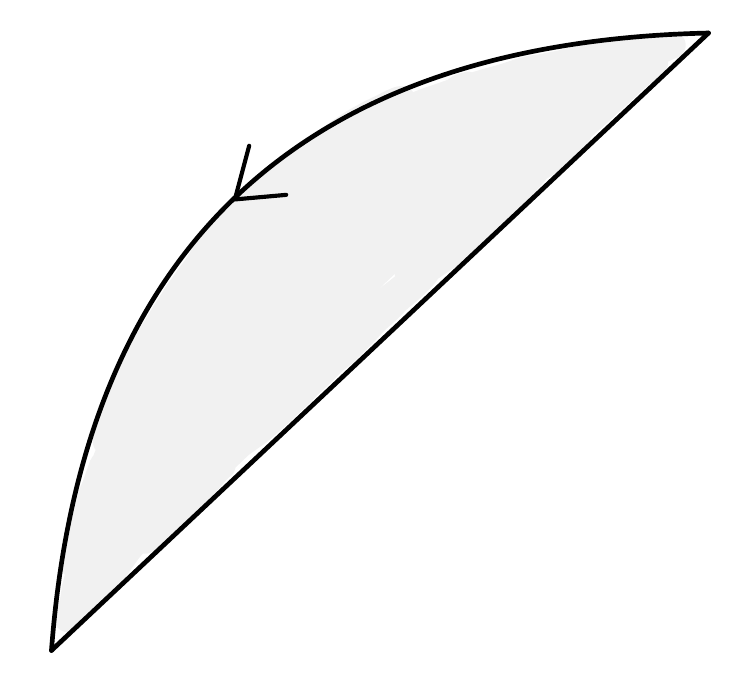}
\end{center}
The quotient $\Omega^\gamma/\l g\r$ is a closed annulus, with a projective structure having one geodesic boundary 
(the image of the geodesic arc) 
and one nondegenerate boundary. 
This is the \emph{half-annulus} associated with $\gamma$. Let $R\in \PGL(3,\R)$ be reflection across the great circle through $\gamma(\pm\infty)$.  Given two curves $\gamma_0,\gamma_1\in\D_3(\CC)$, both with monodromy $g$ and end points $p_\pm$
let $\Omega=\Omega^{\gamma_0}\cup R(\Omega^{\gamma_1})
\subset S^2$. The region $\Omega$  is depicted as 
\begin{center}
	\includegraphics[scale=0.2]{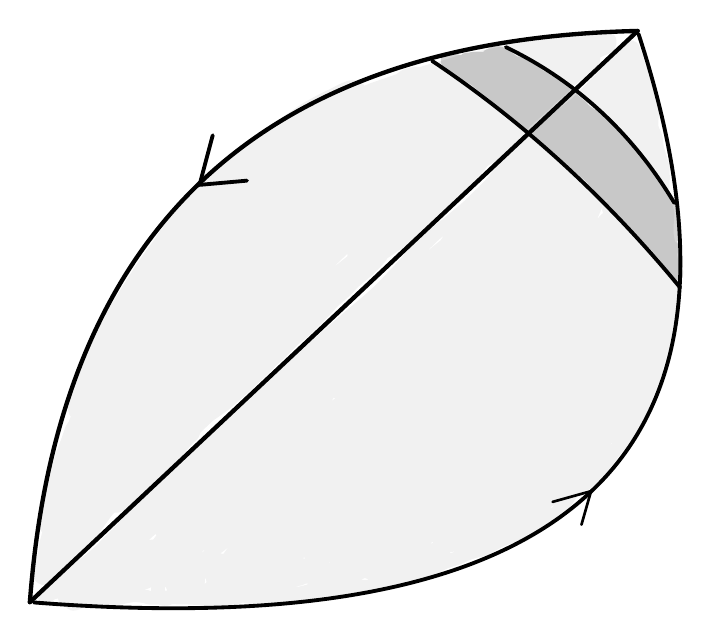}
\end{center}
where the darkly shaded region depicts a fundamental domain for the action of $g$.
The quotient $\Sigma=\Omega/\l g\r$ 
is a projective annulus with nondegenerate boundary. 


 In this way, any pair $(\gamma_1,\gamma_0)\in \D_3^{[2]}(\CC)$ of isomonodromic curves
 with positive hyperbolic monodromy defines a convex projective structure on 
 $\CC\times [-1,1]$. Conjugating such a pair by the action of $G$ does not change the projective structure. 
 We hence obtain an inclusion  
 \[ c\mathfrak{P}(\CC\times [-1,1])\subset \D_3^{[2]}(\CC)/G=\S_3(\CC).\]
 Its image is the restriction of the groupoid $\S_3(\CC)\rra \ca{R}_3(\CC)$ to the subset of operators $L$ with positive hyperbolic monodromy, and corresponding to convex curves $\gamma\in p^{-1}(L)$.  
 Using the $G$-action, we may arrange that the monodromy is in $A_+$, 
 leaving a  residual action of $A$. We hence also have 
 \begin{equation} \label{eq:annulusstructure} c\mathfrak{P}(\CC\times [-1,1])\cong \D^{[2]}_3(\CC)_\Delta/A,\end{equation}
 where $\D^{[2]}_3(\CC)_\Delta\subset \D_3^{[2]}(\CC)$ are pairs of curves in 
 $\D_3(\CC)_\Delta$.   (Cf. Section \ref{subsec:positivehyper}.)  
\end{example}

\begin{example}[Surfaces of negative Euler characteristic]
Let $\Sigma$ be a compact connected oriented surface of genus $\mathsf{g}$ with $\mathsf{r}>0$ boundary components, 
with negative  Euler characteristic  $\chi(\Sigma)=2-2\mathsf{g}-\mathsf{r}<0$. 
Suppose we are given a projective structure in $c\on{Proj}(\Sigma)$, with developing map $\varphi\colon \wt{\Sigma}\to \Omega\subset S^2$ and associated morphism $\varrho\in \Hom(\Gamma,G)$.  
Each boundary circle $\CC\subset \p\Sigma$ is homotopic to a 
unique geodesic loop $\lambda_\CC\subset \Sigma$. Indeed, the restriction 
of $\varphi$ to a component $\wt{\CC}\subset \p\wt{\Sigma}$ mapping to $\CC$ is a nondegenerate curve 
$\gamma\colon \wt\CC\to S^2$ with positive hyperbolic monodromy, and $\lambda_\CC$ is the image of the line segment 
connecting the limit points $\gamma(\pm\infty)$. Let $\Sigma_\CC\subset \Sigma$ be the  principal half-annulus given as the image of $\Omega^\gamma\subset \Omega$. Removing from $\Sigma$ all such half annuli (except for the geodesic itself), we obtain a surface with projective structure having \emph{geodesic boundary}, called the 
\emph{core}
\[ \Sigma_{\on{core}}=\Sigma-\bigsqcup_{\CC\subset \p\Sigma}(\Sigma_\CC-\lambda_\CC).\]
 In the case of a (complete) hyperbolic structure (as in Example \ref{ex:hyperbolic}), this is the usual \emph{convex core} of the surface, and the annuli $\Sigma_\CC$ are the \emph{funnels} (trumpets) -- see for example \cite[Section 2.4.1]{bor:spec}.
 
 The projective structure on $\Sigma_{\on{core}}$ can be transferred back to $\Sigma$, by choice of a diffeomorphism 
 $\Sigma_{\on{core}}\cong \Sigma$ (which we take to be the identity outside collar neighborhood of $\lambda_\CC$). 
 We hence obtain a canonical map 
 \[ c\mf{P}(\Sigma)\to c\mf{P}_{\on{geod}}(\Sigma)\]
to the finite-dimensional deformation space of convex projective structures with geodesic boundary (up to diffeomorphisms isotopic the identity and preserving the boundary, but not necessarily fixing the boundary). 
The latter space comes a natural map to $(A_+)^r$, taking such a projective structure to the conjugacy class of the monodromy around the boundary components $\CC_1,\ldots,\CC_r$. Gluing back the half-annuli, we arrive at the description as a fiber product
over $(A_+)^r$, 
\begin{equation}\label{eq:fiberproduct}
c\mf{P}(\Sigma)=c\mf{P}_{\on{geod}}(\Sigma)\times_{(A_+)^r}\underbrace{ \times_{i=1}^{\mathsf{r}}(\D_3(\CC_i)_\Delta)}.
\end{equation}
Goldman's theory \cite{gol:con}  introduces Fenchel-Nielsen parameters on the space $c\mathfrak{P}_{\on{geod}}(\Sigma)$, for a given 
 pants decomposition of the surface, and a collection of `model seams' \cite[Section 10.6.2]{far:pri}. 
 This involves a copy of $A_+\times A=\R_{>0}^2\times \R^2$ for each of 
 the $3\mathsf{g}+\mathsf{r}-3$ gluing circles, 
 two `internal' parameters for each of the  $2\mathsf{g}+\mathsf{r}-2$ pants, and one 
 copy of $A_+$ for each boundary circle.   By combining this description with \eqref{eq:fiberproduct}, 
 we obtain 
 \[ c\mathfrak{P}(\Sigma)\cong \underbrace{(\R^2)^{2\mathsf{g}+\mathsf{r}-2}}_{\mbox{pants}}\times \underbrace{(\R_{>0}^2\times \R^2)^{3\mathsf{g}+\mathsf{r}-3}}
 _{\mbox{gluing circles}} \times 
 \underbrace{ \times_{i=1}^{\mathsf{r}}(\D_3(\CC_i)_\Delta)}_{\mbox{boundaries}}.\]
\end{example}

\begin{remark}
According to Labourie \cite{lab:fla} and Loftin \cite{lof:aff}, the deformation space of convex projective structures on a compact oriented surface $\Sigma$ without boundary, with $\chi(\Sigma)<0$,  is also the deformation space of 
complex structures $J$  together with a $J$-holomorphic cubic differential. Similarly \cite{ovs:pro}, for $\CC$ a compact oriented 1-manifold, elements of $\ca{R}_3(\CC)$ are equivalent to a projective structures on $\CC$  together with a cubic differential on $\CC$. It  would be interesting to have a combination of these results, for projective structures  with nondegenerate boundary. 
\end{remark}

\subsection{Groupoid action}\label{subsec:groupoidaction}
Given a developing map $\varphi\colon \wt\Sigma\to S^2$ with associated monodromy $\varrho\in \Hom(\Gamma,G)$, let $[\varphi]\in \mf{P}(\Sigma)$ denote its equivalence class modulo the action of $G$ and the action of $\on{Diff}_\oz(\Sigma,\p\Sigma)$. We obtain a groupoid $\L(\Sigma)\rra \mf{P}(\Sigma)$, with arrows 
 \[ [\varphi_1]\longleftarrow [\varphi_0]\]
given by isomonodromic deformations  of developing maps, i.e., homotopies $\varphi_t$ of developing maps for fixed 
$\varrho$.  Since any such  homotopy defines an isomonodromic deformation $\partial \varphi_t$ of the boundary restrictions, we obtain a morphism of groupoids 
\[ 
\begin{tikzcd}
[column sep={5.5em,between origins},
row sep={4em,between origins},]
\L(\Sigma)   \arrow[d,shift left]\arrow[d,shift right]\arrow[r]   & \ca{S}_3(\p\Sigma)    \arrow[d,shift left]\arrow[d,shift right] \\
\mf{P}(\Sigma)\arrow [r]  & \ca{R}_3(\p\Sigma)
\end{tikzcd}
\]
Here $\S_3(\p\Sigma)\rra \ca{R}_3(\p\Sigma)$ denotes the direct product of the groupoids $\S_3(\CC)\rra \ca{R}_3(\CC)$ as $\CC$ ranges over boundary components of $\Sigma$. 

We may think of $\L(\Sigma)$ as the action groupoid for a \emph{local} action of the groupoid $\S_3(\p\Sigma)$. 
To have a global groupoid action, the map 
\begin{equation}\label{eq:action} \L(\Sigma)\to \S_3(\p\Sigma)\times_{\ca{R}_3(\p\Sigma)}\mf{P}(\Sigma).\end{equation}
(given by boundary restriction together with the source map for $\L(\Sigma)$)  would have to be a bijection. 
This map is locally $1-1$, due to the fact that `small' homotopies of the boundary values $\p\varphi$ can be extended to 
homotopies of $\varphi$. But globally, it fails to be surjective in general: For example, $\p\varphi$ could be homotopic to a nondegenerate quasi-periodic map that is no longer injective. This problem does not arise for \emph{convex} projective structures
with positive hyperbolic boundary monodromies:
\begin{proposition}
The restriction of $\L(\Sigma)$ to the deformation space $c\mf{P}(\Sigma)$  is an action groupoid. 	That is, 
\eqref{eq:action} restricts to a bijection over $c\mf{P}(\Sigma)$.
\end{proposition}
\begin{proof}	
Consider a convex projective structure on $\Sigma$, with developing map $\varphi\colon\wt\Sigma\to S^2$ and monodromy $\varrho\in \Hom(\Gamma,G)$, with image $\Omega=\phi(\wt\Sigma)$. 
For each boundary component $\CC$, with a choice of pre-image $\wt{\CC}$, the developing map restricts to a nondegenerate quasi-periodic curve $\gamma\colon\wt\CC\to S^2$ 
with positive hyperbolic monodromy $g_\gamma\in G$. The end points $\gamma(\pm\infty)$ are fixed points of $g_\gamma$, and hence are unchanged under homotopies of 
the boundary values $\p\varphi=\varphi|_{\p\wt\Sigma}$. This implies that $\Omega_{\on{core}}=\varphi(\Sigma_{\on{core}})$
is unchanged under any isomonodromic deformation of $\varphi$. The full region $\Omega$ is the union of 
$\Omega_{\on{core}}$ with the regions  $\Omega^\gamma\subset \Delta^\gamma\subset S^2$, and an 
isomonodromic deformation of $\varphi$ is equivalent to an isomonodromic deformation of the union of all boundary curves 
$\gamma$. 
\end{proof}

The groupoid action on $c\mf{P}(\Sigma)$ may also be seen in terms of the description \eqref{eq:fiberproduct}, 
and is given by the action of 
\begin{equation}\label{eq:subgroupoid} \D^{[2]}_3(\CC)_\Delta/A\subset \S_3(\CC)
\end{equation}
on $\D_3(\CC)_\Delta$, for each boundary component $\CC$.

\section{Symplectic structure}\label{sec:symplectic}
Throughout this section, $\Sigma$ denotes a compact, connected, oriented surface with boundary $\p\Sigma$ (possibly empty). 
The symplectic structure on $\mathfrak{P}(\Sigma)$ will ultimately be obtained from the Atiyah-Bott symplectic structure on the space of connection on a principal bundle, using reduction. Special care is required due to the presence of boundary terms.

\subsection{$\sigma$-positive connections}
Consider a pair $(P,\sigma)$ of a 
principal $G=\SL(3,\R)$-bundle and a $G$-equivariant \emph{developing section}
\begin{equation}\label{eq:sigmamap}
\begin{tikzcd}
[column sep={6.5em,between origins},
row sep={4em,between origins},]
P\arrow[d]\arrow[r,"\sigma"]   & S^2=G/H\\
\Sigma& 
\end{tikzcd}
\end{equation}
Over the boundary, we also fix a lift  of $\p\sigma\colon \p P\to S^2$ to a map (cf.~ \eqref{eq:gammalift}) 
\begin{equation}\label{eq:hattau}
\begin{tikzcd}
[column sep={7.5em,between origins},
row sep={4em,between origins},]
\p P\arrow[d]\arrow[r,"\wh{\p\sigma}"]   &\on{Fl}^+(\R^3)=G/B \\
\p \Sigma& 
\end{tikzcd}
\end{equation}
  Since $\sigma$ is equivalent to a reduction of structure group  to $H\subset G$, the choice of $(P,\sigma)$ amounts to a choice of principal $H$-bundle $P_H$. 
Since $\pi_1(H)=\pi_1(\SO(2))=\Z$ we see that for $\p\Sigma=\emptyset$, 
these are classified by an integer (\emph{Euler number}) 
\begin{equation}
e(P,\sigma)=e(P_H)\in \Z.
\end{equation}
If $\Sigma$ has non-empty boundary, the choice of $\wh{\p\sigma}$ reduces the structure group of $\p P$
to the contractible group $B$, which determines a class of trivializations of $\p P_H\supset \p P_B$ and hence 
$e(P,\sigma)$  is defined. Here, and in what follows, 
we omit reference to the lift \eqref{eq:hattau}, for ease of notation.
Given the data \eqref{eq:sigmamap}, we define: 
\begin{definition}\label{def:positive}
	A connection $\theta\in \Omega^1(P,\g)$ is called \emph{$\sigma$-positive} if: 
	\begin{enumerate}
		\item The horizontal foliation is transverse to $\sigma$-level sets,  
		$\ker\theta\cap \ker(T\sigma)=0$, and the resulting bundle map 
		\[ T\sigma|_{\ker(\theta)}\colon \ker(\theta)\to T S^2\] is orientation preserving. 
		\item The boundary connection $\p\theta\in \Omega^1(\p P,\g)$ is $\p\sigma$-positive (cf.~ Definition \ref{def:taupositive}).
	\end{enumerate}	
The space of $\sigma$-positive connections will be denoted $\A^\sigma(P)$. 
\end{definition}

Note that $\sigma$-positivity is an open condition: If a connection is $\sigma$-positive, then so are its small perturbations.
The pullback of connection to the boundary give a map 
\[ \A^\sigma(P)\to \A^{\wh{\p\sigma}}(\p P).\]
Let $\chi(\Sigma)$ denote the Euler characteristic of $\Sigma$. 
\begin{proposition}\label{prop:eulernumber}
A necessary condition for $(P,\sigma)$ (with given lift \eqref{eq:hattau} of $\p\sigma$) 
to admit a $\sigma$-positive connection $\theta$ is that 
\begin{equation}\label{eq:eulerchar} e(P,\sigma)=\chi(\Sigma).\end{equation}
\end{proposition}
\begin{proof}
Suppose first that $\p\Sigma=\emptyset$. 	The action of the stabilizer $H$ of $v=(1\colon 0\colon 0)$ on $T_v S^2\cong \R^2$ restricts to the usual rotation action of $\SO(2)\subset H$. 	
Hence, 
\[ e(P,\sigma)=e(P_H)=e(P_H\times_H T_vS^2)\]
(where the last entry is the usual Euler number of an oriented rank 2 vector bundle). If $\theta$ is a $\sigma$-positive connection, then 
$T\sigma$ defines an isomorphism of oriented vector bundles $T\Sigma\to	P_H\times_H T_vS^2$. Hence, in this case 
$e(P,\sigma)=e(T\Sigma)$ is the Euler characteristic. The argument extends to a possibly non-empty boundary: here 
 $e(T\Sigma)$ is defined using 
a homotopy class of trivializations of $T\Sigma|_{\p\Sigma}$ for which the boundary 
direction $T\p\Sigma$ is fixed. (See \cite[Proposition 4.9]{al:symteich} for details in the hyperbolic case.) 
\end{proof}

\begin{remark}
	For the group $G'=\on{PSU}(1,1)$ acting on the closed Poincar\'{e} disk $\ol{\DD}$, one may similarly consider principal $G'$-bundles $P'$ with 
	equivariant maps
	\begin{equation}\label{eq:sigmamap2}
	\xymatrix{
		P'  \ar[r]^{\sigma'}\ar[d] & \ol{\DD}\\
		\Sigma & 
	}
	\end{equation}
	(We require that $\sigma'$ is a morphism of manifolds with boundary, in the sense that 	
	the pullback of a boundary defining function for $\p\DD$ is a boundary defining function for $\p P'$.)  
	These are again classified by integers $e(P',\sigma')\in \Z$ (see \cite{al:symteich}). Using the 
	inclusion $\ol{\DD}\hra S^2$ from the 
	Klein-Beltrami model, the pair $(P',\sigma')$ lifts to a pair $(P,\sigma)$, with $e(P',\sigma')=e(P,\sigma)$. 
\end{remark}

\subsection{Projective structures from flat connections}
Fix $(P,\sigma)$, together with the lift $\wh{\p\sigma}$,  
satisfying $e(P,\sigma)=\chi(\Sigma)$.  We shall denote by $\Gau(P,\sigma)$ the group of gauge transformations preserving $\sigma$, and such that the boundary restriction preserves the homotopy class of the lift $\wh{\p\sigma}$. The action of 
$\Gau(P,\sigma)$ on connections preserves the subspace of $\sigma$-positive connections. 

Given a \emph{flat} principal connection $\theta$, the first condition in Definition \ref{def:positive}  means that $\sigma$ restricts to a local diffeomorphism from the $\theta$-horizontal leaves to $S^2$. As explained in Section \ref{subsec:gxstructures},
this determines a projective structure on the $\theta$-horizontal foliation, which then descends to $\Sigma$.  If $\Sigma$ has non-empty boundary, the second condition in Definition \ref{def:positive}  ensures that the boundary will be nondegenerate. This gives a canonical map 
\begin{equation}\label{eq:quotientmap}
\A^\sigma_{\on{flat}}(P) \to \on{Proj}(\Sigma).
\end{equation}
\begin{proposition}\label{prop:firstquotient}
	The group $\Gau(P,\sigma)$ acts freely on $\A^\sigma_{\on{flat}}(P)$, with 
	\[ \on{Proj}(\Sigma)=\A^\sigma_{\on{flat}}(P)/\Gau(P,\sigma).\]
\end{proposition}
\begin{proof}
	If two connections $\theta,\theta'$ are related by a gauge transformation in $\Gau(P,\sigma)$, the resulting projective structure does not change. This shows that \eqref{eq:quotientmap} is constant along  $\Gau(P,\sigma)$-orbits. As explained above, 
	every projective structure with nondegenerate boundary determines a triple $(P_1,\sigma_1,\theta_1)$, where $e(P_1,\sigma_1)=\chi(\Sigma)$. In particular, there is an isomorphism of principal $H$-bundles 
	$(P_1)_H\to P_H$.  (If the boundary is non-empty, then the  isomorphism can be chosen to be compatible with the trivializations along the boundary.) In turn, it  determines an isomorphism $P_1\to P$ taking $\sigma_1$ to $\sigma$. This shows that the map \eqref{eq:quotientmap} is surjective.   
	
	It remains to show that the action is free. By choosing projective charts, it suffices to prove this claim for the standard projective structure on an open subset $U\subset S^2$, with $P=U\times G$, $\sigma(m,g)=g^{-1}\cdot m$ and with the 
	trivial connection, given by the connection 1-form $\mathsf{A}=0$. Suppose a gauge transformation $\psi\in C^\infty(U,G)$ 
	preserves $\sigma$, and satisfies $\psi\bullet A=0$. The second condition means that  $\psi$ is a \emph{constant} gauge transformation, given by $(m,g)\mapsto (m,ag)$ for some \emph{fixed} $a\in G$. But then $(\psi\cdot\sigma)(m,g)=
	\sigma(m,a^{-1}g)=g^{-1}\cdot am$ equals $\sigma(m,g)$ for all $m,g$ if and only if $a\cdot m=m$ for all $m\in U$. That is, $a=e$. 
\end{proof}
For a trivial bundle $P=\Sigma\times G$, principal connections $\theta$ are described in terms of their connection 1-forms $\mathsf{A}$, while  
$\sigma(m,g)=g^{-1}\cdot f(m)$ for a map $f\colon \Sigma\to S^2$.  If $f(m)=\psi(m)\cdot (1\colon 0\colon 0)$ for some $\psi\colon \Sigma\to G$, we can use a gauge transformation by $\psi^{-1}$ to 
replace $f$ with the constant function to the base point $(1\colon 0\colon 0)\in S^2$.

 \begin{example}
 	The tautological projective structure on $S^2$ is described by the trivial bundle $S^2\times G$ and trivial connection $\mathsf{A}=0$, with 
 	$\sigma(m,g)=g^{-1}\cdot m$. On the affine chart $\Sigma=\{(1:u:v)|\ u,v\in \R\}$, we have 
 	$m=\psi(m)\cdot (1\colon 0\colon 0)$ with 
 	\[ \psi(u,v)=\left(\begin{array}{ccc} 1 & 0 &0 \\ u&1&0\\ v & 0& 1
 	\end{array}\right).\]
   After gauge transformation by $\psi^{-1}$, we see that the projective structure is also described by 
   $\sigma(m,g)=g^{-1}\cdot (1\colon 0\colon 0)$ and 
 	\begin{equation}\label{eqref:A}
 	\mathsf{A}=\psi^{-1}\d\psi=\left(\begin{array}{ccc} 0 & 0 &0 \\ \d u&0&0\\ \d v & 0 & 0
 	\end{array}\right).
 	\end{equation}	
 \end{example}

\begin{example}[Principal annulus]
We use the notation from from Example \ref{ex:principalannulus}.  By construction, the annulus $\Sigma$ is a quotient of 
the plane $\wt{\Sigma}=\R\times \R$, with coordinates $(x,y)$, with developing map 
$\varphi\colon \wt{\Sigma}\to S^2$ given by 
\[\varphi(x,y)= (e^{xc_1}:-e^{xc_2}y:e^{xc_3})\]
This may be written as $\varphi=\psi\cdot (1\colon 0\colon 0)$ 
with 
\[\psi(x,y)=\left(\begin{array}{ccc} e^{xc_1}& 0 &0 \\ 0&e^{xc_2}&0\\ 0 & 0 & e^{xc_3}
\end{array}\right)
\left(\begin{array}{ccc} 1 & 0 &0 \\ -y&1&0\\ 1 & 0& 1
\end{array}\right).\]
After gauge transformation by $\psi^{-1}$, we obtain the 
connection 1-form $\mathsf{A}=\psi^{-1}\d\psi$, 
\begin{equation}\label{eq:Aforannulus}
\mathsf{A}=
\left(\begin{array}{ccc} c_1& 0 &0 \\ -y(c_2-c_1)&c_2&0\\ (c_3-c_1) & 0 & c_3
\end{array}\right)\d x-
\left(\begin{array}{ccc} 0 & 0 &0 \\ 1&0&0\\ 0 & 0 & 0
\end{array}\right)\ \d y.
\end{equation}
\end{example}

\begin{proposition}
	Let $P=\Sigma\times G$ with developing section $\sigma(m,g)=g^{-1}\cdot (1\colon 0\colon 0)$. A connection 
	1-form $\mathsf{A}\in \Omega^1(\Sigma,\g)$ is $\sigma$-positive if and only if 
	\[ \mathsf{A}_{21}\wedge \mathsf{A}_{31}\in \Omega^2(\Sigma)\]
	is a volume form, compatible with the given orientation. 
\end{proposition}
\begin{proof}
	In terms of $T_{(m,e)}P=T_m\Sigma\times\g$, the horizontal lift is given by $v\mapsto (v,-\iota_v(\mathsf{A}))$. On the other hand, the 
	differential of the map $G\to S^2,\ g\mapsto g^{-1}\cdot (1\colon 0\colon 0)$ at $e$ takes  $\xi\in \g$ to 
	$-(\xi_{21},\xi_{31})^\top\in \R^2$ (the last two entries from the first column of $\xi$). Hence, the condition is that 
	$v\mapsto \iota(v)(\mathsf{A}_{21},\mathsf{A}_{31})^\top$ is an orientation preserving isomorphism. 
	Equivalently, $\mathsf{A}_{21}\wedge \mathsf{A}_{31}$ is a volume form. 
\end{proof}

\begin{example}
Let $L\in \ca{R}_3(S^1)$ be a differential operator  $L=\partial^3+a_1\partial+a_0$ with coefficient functions $a_0, a_1$. Consider the half-annulus $S^1\times [0,\infty)$ with coordinates $u,v$. 
The connection 1-form 
\[ \mathsf{A}=\left(\begin{array}{ccc} 1 & 0 &0 \\ 0&1&0\\ -v & 0 & 1
\end{array}\right)\bullet 
\left(\begin{array}{ccc} 0&0&-a_0\\ 1&0& -a_1\\ 0  & 1 & 0
\end{array}\right)\d u\]
(where $\bullet$ indicates a gauge transformation, $g\bullet \mathsf{A}=\Ad_g \mathsf{A}-\d g g^{-1}$)
is flat, and is $\sigma$-positive near the boundary. (A calculation shows that $\mathsf{A}_{21}\wedge \mathsf{A}_{31}=(1+v a_0(u))\d u\wedge\d v$.) 
Its pullback to the boundary is the Drinfeld-Sokolov normal form of the connection corresponding to $L$.  
\end{example}

\subsection{Construction of the symplectic structure}
For the time being, let $G$ be any Lie group admitting an $\Ad$-invariant nondegenerate symmetric bilinear form $\cdot$ on its Lie algebra $\g$. For every principal $G$-bundle  $P\to \Sigma$ over a compact oriented surface (possibly with boundary), the choice of this bilinear form defines the \emph{Atiyah-Bott symplectic structure} \cite{at:mo} on the space $\A(P)$ of principal connections. 
Identifying the tangent spaces to $\A(P)$ with 1-forms valued in the adjoint bundle $\g(P)=P\times_G\g$, the Atiyah-Bott form is given by wedge product followed by integration:  $\int_\Sigma \alpha\cdot \beta$ for $\alpha,\beta\in \Omega^1(\Sigma,\g(P))$. 
 The symplectic structure is preserved under the action 
 \[  \on{Aut}_+(P)\circlearrowright \A(P)\]
 of the group of principal bundle automorphisms with orientation preserving base map. The action admits an affine moment map, 
 with a bulk contribution involving the curvature $F^\theta\in \Omega^2(\Sigma,\g(P))$, and a boundary contribution involving the pullback $\partial\theta$ of the connection. In more detail, identify the Lie algebra $\mf{aut}(P)$ with sections of the Atiyah algebroid $\on{At}(P)=TP/G$, and recall that  a connection defines a splitting $s^\theta\colon \on{At}(P)\to \g(P)$. Then  the  moment map component for $\zeta\in \mf{aut}(P)$ is given by 
 %
 \begin{equation}\label{eq:gaugemoment}
 \A(P)\to \R,\ \ \theta\mapsto -\int_\Sigma F^\theta\cdot s^\theta(\zeta)+\int_{\p\Sigma}\mbox{bdry terms}.
 \end{equation}
 The boundary terms depend only on $\p\theta$ and $\zeta|_{\p\Sigma}$, and 
 vanish if $\zeta$ vanishes along $\p\Sigma$.    
  See \cite[Appendix B]{al:symteich} for a detailed discussion. 
  The Atiyah-Bott structure on connections descends to a symplectic structure on 
 \begin{equation}\label{eq:loopgroupspace}
 \M(P)=\A_{\on{flat}}(P)/\Gau(P,\p P),
 \end{equation}
 the moduli space  of flat connections, up to  gauge transformations that are trivial along the boundary. This space 
 carries a  residual gauge action of  $\Gau_\oz(\p P)$, with moment map $\M(P)\to \A(\p P),\ [\theta]\mapsto \p\theta$. 
 
In our case, where  $G=\SL(3,\R)$, we take the nondegenerate symmetric bilinear form on $\g=\mf{sl}(3,\R)$ to be the trace form $(\xi,\eta)\mapsto \on{tr}(\xi\eta)$.  
 By Proposition \ref{prop:firstquotient}, we have that  $\on{Proj}(\Sigma)$ is the quotient of 
 $\A^\sigma_{\on{flat}}(P)$ under $\Gau(P,\sigma)$, and 
 $\mathfrak{P}(\Sigma)$ is obtained by taking another quotient by $\Diff_\oz(\Sigma,\p\Sigma)$. 
 This  suggests a construction of the symplectic form on $\mathfrak{P}(\Sigma)$ via reduction with respect to 
 the subgroup of  automorphisms preserving $\sigma$, and with 
 base map in $\Diff_\oz(\Sigma,\p\Sigma)$. However, if $\p\Sigma$ has non-empty boundary, this is not  directly a symplectic quotient: 
 automorphisms whose base map fixes the boundary might still act non-trivially on $P|_{\p\Sigma}$, and 
 hence will produce boundary terms in the moment map. 
 
 Following the strategy in \cite{al:symteich}, we introduce an intermediate space
\begin{equation}\label{eq:fiberproduct}
\M^\sigma(P)=\A^\sigma_{\on{flat}}(P)/\Aut_\oz(P,\p P,\sigma)
\end{equation}
where  $\Aut_\oz(P,\p P,\sigma)\subset \Aut_+(P,\sigma)$ are the automorphisms
preserving $\sigma$, with base map homotopic to the identity, and fixing 
$\p P$ \emph{pointwise}. This space carries a residual action of $\Gau_\oz(\p P,\p\sigma)$, and an equivariant map to 
$\A^{\p\sigma}(\p P)\subset \A(\p P)$. 

\begin{proposition}\label{lem:hatdef}
	The space \eqref{eq:fiberproduct} is a symplectic reduction, 
	\begin{equation}\label{eq:hatdef}
	\M^\sigma(P)=\A^\sigma(P)\qu \Aut_\oz(P,\p P,\sigma).\end{equation}
The inclusion of $\sigma$-positive flat connections descends to a $\Gau_\oz(\p P,\p\sigma)$-equivariant local 
symplectomorphism,  
\begin{equation}\label{eq:amap} \M^\sigma(P)\to \M(P)\end{equation}
%
intertwining the maps to $ \A(\p P)$. 
\end{proposition}
\begin{proof}
	The Lie algebra of $\on{Aut}(P,\sigma)$ are the $G$-invariant vector fields on $P$ preserving $\sigma$, or equivalently are tangent to $P_H$. The vector fields are the sections of a sub-Lie algebroid $\on{At}(P,\sigma)\cong \on{At}(P_H)$ of the Atiyah algebroid $\on{At}(P)=TP/G$. The space 
	$\mf{aut}(P,\p P,\sigma)$ consists of sections $\zeta$ of  $\on{At}(P,\sigma)$ vanishing along $\p\Sigma$. The moment map for $\Aut_\oz(P,\p P,\sigma)$ is obtained from \eqref{eq:gaugemoment} by restriction, and hence is simply 
	\[ \theta\mapsto -\int_\Sigma F^\theta\cdot s^\theta(\zeta),\] 
	not involving a 
	boundary term. 	A connection $\theta$ satisfies the first condition in Definition \ref{def:positive} if and only if the restriction 
	\begin{equation}\label{eq:siso} s^\theta|_{\on{At}(P,\sigma)}\colon \on{At}(P,\sigma)\to \g(P)\end{equation}
	is an isomorphism (cf.~ \cite[Proof of Proposition 6.2]{al:symteich}). We hence see that the integral vanishes for all $\zeta\in \mf{aut}(P,\p P,\sigma)$ if and only if $F^\theta=0$, i.e, $\theta$ is flat. That is, the zero level set for the 
	$\Aut_\oz(P,\p P,\sigma)$-action is exactly $\A^\sigma_{\on{flat}}(P)$, with null foliation given by 
	$\Aut_\oz(P,\p P,\sigma)$-orbit directions. 
	
On the other hand, the null foliation of $\A_{\on{flat}}(P)$ is also given by the $\Gau(P,\p P)$-orbit directions. It follows that locally, the 
	$\Aut_\oz(P,\p P,\sigma)$-orbit directions and  $\Gau(P,\p P)$-orbit directions coincide. 
	Every element 
	of  $\Aut_\oz(P,\p P,\sigma)$ is homotopic to an element of $\Gau(P,\p P,\sigma)$, by deforming the base map. Thus, if two $\sigma$-positive connections are in the same $\Aut_\oz(P,\p P,\sigma)$-orbit then they are also in the same 
	$\Gau(P,\p P)$-orbit. 
	Hence, the inclusion of $\sigma$-positive connections 
	descends to a local symplectomorphism \eqref{eq:amap}. 
\end{proof}

The deformation space of projective structures is obtained as the quotient 
 \begin{equation}\label{eq:defsigma}
\mathfrak{P}(\Sigma)=\M^\sigma(P)/\Gau_\oz(\p P,\p\sigma).
\end{equation}
The symplectic structure does not directly descend since it is not $\Gau_\oz(\p P,\p\sigma)$-basic. 
However, note that the moment map for the action of $\Gau_\oz(\p P,\p\sigma)$ is given by the map 
$\M^\sigma(P)\to \A(\p P),\ \ [\theta]\mapsto \p\theta$ followed by the quotient map to $\A(\p\P,\p\sigma)$:
\begin{equation}\label{eq:pt}
 \M^\sigma(P)\to  \A(\p P)/\on{ann}(\gau(\p P,\p\sigma)),\ \ [\theta]\mapsto [\p\theta]
\end{equation} 
Here the brackets indicate equivalence classes of connections. By Proposition \ref{prop:singleorbit}, the 
 image of \eqref{eq:pt} is a single coadjoint $\Gau(\p P,\p\sigma)$-orbit 
 \begin{equation}\label{eq:O}
 \O=\A^{\p\sigma-\on{reg}}(\p P)/\on{ann}(\gau(\p P,\p\sigma)).
 \end{equation}
 Using the `shifting trick' from symplectic geometry, we may therefore write \eqref{eq:defsigma} as a symplectic quotient
 with respect to $\O$. To summarize: 
\begin{proposition}\label{prop:shiftingtrick} 
	The space $\mathfrak{P}(\Sigma)$ is a 
	 symplectic quotient, 
	\begin{equation}\label{eq:symquot}
	\mathfrak{P}(\Sigma)=	\mathfrak{P}(\Sigma)_\O=	
	(\M^\sigma(P)\times \O^-)\qu \Gau_\oz(\p P,\p\sigma).
	\end{equation}
\end{proposition} 

We use Equation \eqref{eq:symquot} to define a symplectic structure on $\mathfrak{P}(\Sigma)$. 
\begin{remark}
If $\p\Sigma=\emptyset$, the second step is not needed, and we directly have 
\[ 	\mathfrak{P}(\Sigma)=\A^\sigma(P)\qu \Aut_\oz(P,\sigma).\]
\end{remark}

\begin{remark}
For the case $\p\Sigma=\emptyset$, Goldman \cite{gol:symaff} gives an alternative construction of the symplectic structure on 
$\mf{P}(\Sigma)$ as an infinite-dimensional symplectic quotient. It would be interesting to generalize Goldman's approach to the setting with nondegenerate boundary. 
\end{remark}


\section{Moment map property}\label{sec:momentmap}
Recall some notions from symplectic geometry.  A \emph{Lagrangian relation} $L\colon M_1\da M_2$ between symplectic manifolds $M_1,M_2$ is a Lagrangian submanifold $L\subset M_2\times M_1^-$, where the superscript indicates the opposite symplectic structure. The action of a symplectic groupoid $\S\rra B$ on a symplectic manifold $M$, along a
map $\Phi\colon M\to B$, is called \emph{Hamiltonian} if the action map is a Lagrangian relation 
\[ L\colon S\times M \da M.\] 
The notion of a Hamiltonian action of a symplectic groupoid is due to Mikami-Weinstein \cite{mik:mom}, who observed that ordinary Hamiltonian actions of Lie groups $G$ on symplectic manifolds are equivalent to Hamiltonian actions of the cotangent groupoid $T^*G\rra \g^*$. 

In Section \ref{subsec:groupoidaction}, we described a  `local' groupoid action of  $\S_3(\p\Sigma)\rra \ca{R}_3(\p\Sigma)$
on the deformation space $\mathfrak{P}(\Sigma)$ of projective structures 
along the map 
$\Psi\colon \mathfrak{P}(\Sigma)\to \ca{R}_3(\p\Sigma)$. This `local' action was defined in terms of 
a subgroupoid
\[ \L(\Sigma)\subset \mf{P}(\Sigma)\times (\S_3(\p \Sigma)\times \mf{P}(\Sigma))^-.\]
We will now show that the local groupoid action is Hamiltonian, in the sense that $\L(\Sigma)$ 
defines a Lagrangian relation.  We shall prove this by comparing with similar groupoids for the spaces $\M(P)$ and $\M^\sigma(P)$. 
Our strategy is described by the diagram 
\[ 
\begin{tikzcd}
[column sep={8.5em,between origins},
row sep={4em,between origins},]
\S(\p P)\times \M(P) \arrow[r,dashed,"\L(P)"]   & \M(P)\\
\S^{\p\sigma}(\p P)\times \M^\sigma(P) \arrow[u]\arrow[d]\arrow[r,dashed,"\L^\sigma(P)"]   & \M^\sigma(P)\ar[u]\ar[d]\\
\S_3(\p \Sigma)\times \mf{P}(\Sigma) \arrow[r,dashed,"\L(\Sigma)"]   &  \mf{P}(\Sigma)
\end{tikzcd}
\]
where the dotted arrows are Lagrangian relations.
In the diagram, $\L(P)$ is given by the action of the symplectic groupoid $\S(\p P)\rra \A(\p P)$ on $\M(P)$. 
The middle arrow is obtained from the top row by `restriction'. The bottom row is obtained from the  middle row by Drinfeld-Sokolov reduction.

\subsection{The groupoids $\S(Q),\S^{\wh{\tau}}(Q),\S_3(\CC)$}
Let $Q\to \CC$ be a principal $G$-bundle over an abstract circle. The space $\A(Q)$ of principal connections is an infinite-dimensional 
Poisson manifold, with symplectic leaves the orbits of the identity component of the gauge group $\Gau_\oz(Q)$.  
(An isomorphism $Q\cong S^1\times G$ identifies the gauge group with the loop group.)  
Its symplectic groupoid is the action groupoid  
\[ \S(Q)=\Gau_\oz(Q)\times \A(Q)\rra \A(Q),\]
with 2-form $\omega_\S$ obtained from its identification with the (twisted) cotangent bundle of $\Gau_\oz(Q)$.  (See \cite[Example A.12]{al:coad}.)  
It admits an alternative description in terms of a space 
$\P(Q)$ of \emph{quasi-periodic sections} of the pullback bundle $\wt{Q}\to \wt{\CC}$. Here $t\in \Gamma(\wt{Q})$ is 
called quasi-periodic if it satisfies $\kappa^*t=\mu(t)\cdot t$ for some \emph{monodromy} $\mu(t)\in G$.
The principal action of $G$ on $P$ induces an action of $G$ on quasi-periodic sections, with quotient map 
\[ p\colon \P(Q)\to \A(Q)\] 
taking $t\in\P(Q)$  to the unique connection $\theta$  for which $t$  is parallel. Let $\P^{[2]}(Q)$ consist of pairs $(t_1,t_0)$ of 
quasi-periodic sections that are related by isomonodromic homotopy. 
Equivalently, the unique gauge transformation taking $t_0$ to $t_1$ lies in $\Gau_\oz(Q)$. 
We hence have $\P^{[2]}(Q)\cong \Gau_\oz(Q)\times \P(Q)$, which descends to an identification 
\begin{equation}\label{eq:alternatedescription} \S(Q)\cong \P^{[2]}(Q)/G\rra \A(Q).\end{equation}
As shown in \cite{al:coad}, the space $\P(Q)$ carries a distinguished $\Aut(Q)\times G$-invariant 2-form $\varpi_\P$, with nice properties.  Furthermore, this 2-form establishes a Morita equivalence of $\S(Q)$ with a quasi-symplectic groupoid over the universal cover of $G$.  The following is a direct consequence of this result: 
\begin{proposition}
	The pullback of $\omega_\S$ to $\P^{[2]}(Q)$ is the difference 	$\pr_1^*\varpi_\P-\pr_2^*\varpi_P$, where 
	$\pr_i\colon \P^{[2]}(Q)\to \P(Q)$ are the two projections. 
\end{proposition}

Given $\tau\colon Q\to S^2$ with a lift $\wh{\tau}\colon Q\to \on{Fl}^+(\R^3)$, we can consider a related space 
$\P^{\wh{\tau}}(Q)$ of quasi-periodic sections such that the corresponding connection $\theta$ is $\wh\tau$-positive. (Cf. Definition \ref{def:taupositive}.)  
Let $\P^{\wh\tau,[2]}(Q)$ consist of pairs of such sections, which furthermore are related by an isomonodromic homotopy 
of such sections. Then 
\begin{equation}\label{eq:alternatedescription1} \S^{\wh{\tau}}(Q)\cong \P^{\wh\tau,[2]}(Q)/G\rra \A^{\wh{\tau}}(Q)\end{equation}
is a subgroupoid of $\S(Q)$. Since $\wh{\tau}$-positivity is an open condition, this subgroupoid is open in $\S(Q)$, and in particular is symplectic. The group $\Gau_\oz(Q,\tau)$ acts on $\A^{\wh{\tau}}(Q)$, with quotient $\ca{R}_3(\Sigma)$. Likewise, the quotient of $\S^{\wh{\tau}}(Q)$ under the action of $\Gau_\oz(Q,\tau)\times \Gau_\oz(Q,\tau)$ is the symplectic groupoid 
$\ca{S}_3(\Sigma)\rra \ca{R}_3(\Sigma)$. Similar to the discussion before Proposition \ref{prop:shiftingtrick}, the symplectic structure does not descend, but instead is obtained via the shifting trick. That is, 
\begin{equation}
 \ca{S}_3(\Sigma)=\big(\S^{\wh{\tau}}(Q)\times (\O^-\times \O)\big)\qu \big(\Gau_\oz(Q,\tau)\times \Gau_\oz(Q,\tau)\big).
\end{equation}
(In this last equality, we may replace $\S^{\wh{\tau}}(Q)$ with $\S(Q)$.) For more details on this construction of the symplectic groupoid, see \cite{kha:1}.

\subsection{The (local) groupoid actions}
Returning to our setting from Section \ref{sec:symplectic}, consider a compact oriented surface $\Sigma$ with boundary, a principal $G$-bundle $P\to \Sigma$, a developing section $\sigma\colon P\to S^2$, and a lift 
$\wh{\p\sigma}$ of the boundary restriction such that $e(P,\sigma)=\chi(\Sigma)$. We also fix a universal cover 
$\wt{\Sigma}\to \Sigma$ with group $\Gamma$ of deck transformations; the pullback bundle is denoted $\wt{P}\to \wt{\Sigma}$.  

The symplectic groupoid
\[ \S(\p P)=\Gau_\oz(\p P)\times \A(\p P)\rra \A(\p P)\] 
is  the direct product over all $\S(\p P|_{\CC_i})\rra \A(\p P|_{\CC_i})$ as $\CC_i$ ranges over all components of $\p\Sigma$. 
As in \eqref{eq:alternatedescription}, it may be described in terms of pairs $(t_1,t_0)$ of $\Gamma$-quasi-periodic sections 
$t\colon \p \wt\Sigma\to \p\wt{ P}$. 

The action of $\Gau_\oz(\p P)$ on the moduli space $\M(P)$ is equivalent to  an action of the symplectic groupoid $\S(\p P)$. 
In particular, its graph defines a Lagrangian subgroupoid
\[ \L(P)\subset \M(P)\times \big(\S(\p P)\times \M(P)\big)^-.\]
It is convenient to regard  $\M(P)$ as the quotient, under the action of $\Gau(P,\p P)\times G$, of the space of $\Gamma$-quasi-periodic sections 
$s\colon \wt{\Sigma}\to \wt{P}$. Letting $[s]$ denote the equivalence class of such a section, the subgroupoid $\L(P)$ 
consists of all 
\begin{equation}\label{eq:equivalenceclasses}
([s_1],[(t_1,t_0)],[s_0])
\end{equation}
such that $s_0$ is homotopic to $s_1$ with fixed monodromy, and $t_0,t_1$ are the 
boundary values of $s_0,s_1$. 

The space $\M^\sigma(P)$ (cf. \eqref{eq:fiberproduct}) can similarly be described as 
equivalence classes of sections $s$, with the additional requirement that $s$ is $\sigma$-positive in the sense that 
the connection corresponding to $s$ is 
$\sigma$-positive. As in the previous section, define the (open) symplectic subgroupoid
$\S^{\wh{\p\sigma}}(\p P)\subset \S(\p P)$. Its `local' groupoid action on $\M^\sigma(P)$ is given by the
subgroupoid
\begin{equation}
\L^\sigma(P)\subset \M^\sigma(P)\times \big(\S^{\wh{\p\sigma}}(\p P)\times \M^\sigma(P)\big)^-,
\end{equation}
consisting of equivalence classes \eqref{eq:equivalenceclasses}, but with $s_0,s_1$ being $\sigma$-positive, and the isomonodromic homotopy between them being through $\sigma$-positive  sections. From this description it is evident that the local symplectomorphisms $\M^\sigma(P)\to \M(P)$ and $\S^{\wh{\p\sigma}}(\p P)\subset \S(\p P)$ restrict to a 
map, locally 1-1, $\L^\sigma(P)\to \L(P)$. In particular, $\L^\sigma(P)$ is Lagrangian. 

\begin{theorem}
The subgroupoid 
	\begin{equation}\label{eq:lsigma}
	\L(\Sigma)\subset \wh{\mf{P}}(\Sigma)\times \big(\S_3(\p\Sigma)\times \wh{\mf{P}}(\Sigma)\big)^- 
	\end{equation}
is Lagrangian. That is, $\L(\Sigma)$ defines a (local) Hamiltonian action of a symplectic 
groupoid.  
\end{theorem}
\begin{proof}
The space $\L(\Sigma)$ was 	described in Section \ref{subsec:groupoidaction}
 in terms of equivalence classes $[(\phi_1,\phi_0)]$ of developing maps
$\wt{\Sigma}\to S^2$. Every such developing map comes from a $\sigma$-positive quasi-periodic section 
$s\colon \wt{\Sigma}\to \wt{P}$; two sections give the same developing map if they are related by 
$\Gau(P,\sigma)$. Consequently, $\L(\Sigma)$ is just the image of $\L^\sigma(P)$ under the quotient maps 
\[ \M^\sigma(P)\to {\mf{P}}(\Sigma),\ \ \ \S^{\wh{\p\sigma}}(\p P)\to \S_3(\p\Sigma).\]
As we saw, these quotient maps can also be seen as reductions with respect to the orbit $\O$. We conclude 
that $\L(\Sigma)$ is Lagrangian. 
\end{proof}

\section{Goldman twists}\label{sec:goldmantwists}
Suppose $\lambda\subset \Sigma$ is an oriented simple loop, not homotopic to a boundary 
component. Associated with $\lambda$ is the $A$-action (twist) on the space $c\mathfrak{P}(\Sigma)$, as described in Section
\ref{subsec:twists}. For the case without boundary, it was shown in \cite{gol:con} and \cite{kim:sym} that the $A$-action is 
Hamiltonian, with moment map the logarithm of the holonomy 
\[ \on {Hol}_\lambda\colon c\mathfrak{P}(\Sigma)\to A_+.\]
(The holonomy is defined up to conjugacy, but each positive hyperbolic element is conjugate to a unique element of $A_+$). 
We will verify this fact for our description of the symplectic structure, with a possibly non-empty boundary. We shall use 
\[ \xi_1=\on{diag}(-\f{1}{2},0,\f{1}{2}),\ \ \xi_2=\on{diag}(-\f{1}{3},\f{2}{3},-\f{1}{3})\]
as a basis for the Lie algebra $\mf{a}=\on{Lie}(A)$. Let 
$\on{diag}(e^{c_1},e^{c_2},e^{c_3})\in A_+$ be the monodromy around $\lambda$. 

\begin{theorem}\label{th:goldmantwist}
	Given a simple loop $\lambda$ in $\Sigma$, not homotopic to a boundary component, the function 
	\[ H_i=\on{tr}(\log(\on{Hol}_\lambda)\xi_i)\colon c\mathfrak{P}(\Sigma)\to \R\]
	is a Hamiltonian for the twist flow given by $\xi_i$. 
\end{theorem}
Writing $\on{Hol}_\lambda=(e^{c_1},e^{c_2},e^{c_3})\in A_+$ with 
$ c_1<c_2<c_3,\ \sum c_i=0$, we have 
\[ H_1=\hh (c_3-c_1),\ H_2=c_2.\]
This agrees with \cite[Section 1.8]{gol:con} up to normalization. 

To prove the theorem, fix a pair $(P,\sigma)$ and list $\wh{\p\sigma}$ for the construction of the symplectic structure. 
Near $\lambda$, we may choose a trivialization of $P$ such that $\sigma$ is given by 
$g^{-1}\cdot (1\colon 0\colon 0)$. These data allow us to describe projective structures, near $\lambda$,  
by connections 1-form $\mathsf{A}$. A convex projective structure having $\lambda$ as a closed geodesic is modelled
by the principal annulus, with connection 1-form $\mathsf{A}$ given by Equation  \eqref{eq:Aforannulus}. We shall need 
connection 1-forms describing the images of this projective structure under the twist flows.

Fix  $\epsilon>0$, and 
let $\chi\in C^\infty(\R)$ be a function that is equal to $0$ for $y\le -\epsilon$ and equal to $1$ for 
$y\ge \epsilon$. The maps 
\[ \mathsf{F}_i(t)\colon (x,y)\mapsto \exp(t \chi(y)\xi_i))\cdot (x,y)\] 
are diffeomorphisms for $|t|$ sufficiently small, and represent twists by $\exp(t\xi_i)\in A$. 
\begin{lemma}
	The pullback of the projective structure under these diffeomorphisms corresponds to connection 1-forms 
	$\mathsf{A}_i(t)=\mathsf{A}+t\alpha_i$ with 
	\begin{equation}\label{eq:Aforshear}
	\alpha_1= 
	\left(\xi_1+\left(\begin{array}{ccc}0&0&0\\
	-\f{1}{2}y
	&0&0\\ 1&0&0\end{array}\right)
	\right)\chi'(y)\d y,\ \ \ \ \ 
	\alpha_2=
	\left(\xi_2+\left(\begin{array}{ccc}0&0&0\\
	-y&0&0\\0&0&0\end{array}\right)
	\right)
	\chi'(y)\d y.
	\end{equation}
\end{lemma}
\begin{proof}
	The expressions are computed as $\mathsf{F}_i(t)^*\mathsf{A}$, followed by a  gauge transformation, arranging that 
		$\mathsf{A}_i(t)$ agrees  with $\mathsf{A}$ for $|y|\ge \eps$. The calculation shows that we may take 
	\[ \mathsf{A}_1(t)=\exp(-bt \chi \xi_2)\bullet \mathsf{F}_1(t)^*\mathsf{A},\ \ 
	\mathsf{A}_2(t)=\exp(-t \chi \xi_2)\bullet \mathsf{F}_2(t)^*\mathsf{A}\]
	where $b=\f{1}{2}-\f{c_2-c_1}{c_3-c_1}$. 
\end{proof}

Observe that $\alpha_i$ do not depend on the $c_i$, and that they  
have the special form $\alpha_i=\xi_i\chi'(y)\d y+\mbox{lower triangular}$.

\begin{proof}[Proof of Theorem \ref{th:goldmantwist}]
	The claim is that 
	\begin{equation}\label{eq:mommap}
	-\iota(\xi_i^\sharp)\omega=\d H_i
	\end{equation}
	where the vector field $\xi_i^\sharp$ on $c\mathfrak{P}(\Sigma)$ 
	is the generator of the twist flow defined by $\xi_i$.  	We shall verify this identity at a given convex projective structure, represented in terms of a $\sigma$-positive connection $\theta$.
	The twist flow may be realized in terms of a change of $\theta$ that is localized near the given geodesic $\lambda$. We may hence 
	work in the local model, given by the principal annulus, with the trivial $G$-bundle and $\sigma$ in its standard form, and with the connection 1-form $\mathsf{A}$ given by \eqref{eq:Aforannulus}, and with the twist flow described by the forms 
	$\alpha_i$ from \eqref{eq:Aforshear}. 
	
	Consider now an arbitrary tangent vector to $c\mf{P}(\Sigma)$ at the given projective structure. It is 
	represented by a 1-parameter family of projective structures, and hence is given 
	near $\lambda$ by a 1-parameter 
	family of connection 1-forms $\mathsf{A}_t$, with $\mathsf{A}_0=\mathsf{A}$. We may assume that $\lambda$ (i.e., the submanifold $y=0$) remains a geodesic for all $t$, and that the pullback of $\mathsf{A}_t$ to the submanifold $y=0$ has the standard form  given by the first term in \eqref{eq:Aforannulus}, with the $c_i$ depending on $t$.  Thus, 
	\[ \log\on{Hol}_\lambda(A_t)=\on{diag}(c_1(t),c_2(t),c_3(t))\in\mf{a}.\]
	%
	The $t$-derivative of the family of connections $\mathsf{A}_t$ is of the form 
	\[ \f{d}{d t}|_{t=0}\mathsf{A}_t=(\zeta+\mbox{lower triangular}) \d x+ [\ldots]\d y.\]
	where $\zeta=\f{d}{d t}|_{t=0} \on{diag}(c_1(t),c_2(t),c_3(t))$. Hence, 
	\[ \int \on{tr}(\alpha_i \f{d}{d t}|_{t=0}\mathsf{A}_t)=
	-\int \on{tr}(\xi_i \zeta)\chi'(y)\d x\d y=-\on{tr}(\xi_i \zeta)
	=-\f{d}{d t}|_{t=0} \on{tr}(\xi_i \log\on{Hol}_\lambda(A_t))
	.\]
	This verifies the moment map property \eqref{eq:mommap}. 
\end{proof}

\bibliographystyle{amsplain} 

\def\cprime{$'$} \def\polhk#1{\setbox0=\hbox{#1}{\ooalign{\hidewidth
			\lower1.5ex\hbox{`}\hidewidth\crcr\unhbox0}}} \def\cprime{$'$}
\def\cprime{$'$} \def\cprime{$'$} \def\cprime{$'$} \def\cprime{$'$}
\def\polhk#1{\setbox0=\hbox{#1}{\ooalign{\hidewidth
			\lower1.5ex\hbox{`}\hidewidth\crcr\unhbox0}}} \def\cprime{$'$}
\def\cprime{$'$} \def\cprime{$'$} \def\cprime{$'$} \def\cprime{$'$}
\providecommand{\bysame}{\leavevmode\hbox to3em{\hrulefill}\thinspace}
\providecommand{\MR}{\relax\ifhmode\unskip\space\fi MR }
\providecommand{\MRhref}[2]{%
	\href{http://www.ams.org/mathscinet-getitem?mr=#1}{#2}
}
\providecommand{\href}[2]{#2}

\end{document}